\documentclass[11pt,leqno]{article}
\usepackage{amsmath, amscd, amsthm, amssymb, graphics, xypic, mathrsfs, setspace, fancyhdr, bm, pdfsync, enumitem, mathptmx}
\usepackage[usenames, dvipsnames, svgnames, table]{xcolor}
\usepackage[letterpaper,top=1.05in, bottom=1.05in, left=1.05in, right=1.05in]{geometry}
\usepackage[colorlinks=true,pagebackref=true]{hyperref}
\usepackage[textsize=scriptsize]{todonotes}
\hypersetup{backref,hidelinks}

\usepackage{etoolbox}
\newtoggle{final}
\toggletrue{final}

\iftoggle{final} {
\newcommand{\aravind}[1]{} 
\newcommand{\tom}[1]{} 
\newcommand{\mike}[1]{} 
\newcommand{\NB}[1]{}
\newcommand{\TODO}[1]{}
\renewcommand{\todo}[1]{}
}{ 
\newcommand{\aravind}[1]{\textcolor{red}{#1}} 
\newcommand{\tom}[1]{\todo[color=blue!40]{#1}} 
\newcommand{\mike}[1]{\textcolor{green}{#1}} 
\newcommand{\NB}[1]{\todo[color=gray!40]{#1}}
\newcommand{\TODO}[1]{\todo[color=red]{#1}}
}

\newcommand{\colim}{\operatorname{colim}}

\newcommand{\Spec}{\operatorname{Spec}}

\newcommand{\weq}{\simeq}

\newcommand{\sma}{{\scriptstyle{\wedge}\,}}

\newcommand{\op}[1]{\operatorname{#1}}

\renewcommand{\hom}{\operatorname{Hom}}
\newcommand{\Map}{\operatorname{Map}}
\newcommand{\iMap}{\underline{\Map}}

\newcommand{\cplx}{{\mathbb C}}

\newcommand{\Z}{{\mathbb Z}}

\newcommand{\A}{{\mathbb A}}

\newcommand{\aone}{{\mathbb A}^1}
\newcommand{\pone}{{\mathbb P}^1}

\newcommand{\gm}[1]{{{\mathbb G}_{m}^{#1}}}

\newcommand{\ho}[2][]{\mathrm{Spc}_{#1}({#2})}

\newcommand{\bpi}{\bm{\pi}}

\newcommand{\Nis}{{\operatorname{Nis}}}
\newcommand{\Zar}{\operatorname{Zar}} 
\newcommand{\SH}{{\mathrm{SH}}}

\newcommand{\Shv}{{\mathrm{Shv}}}
\newcommand{\Sm}{\mathrm{Sm}}
\newcommand{\Sch}{\mathrm{Sch}}

\newcommand{\Spc}{\mathrm{Spc}}

\newcommand{\Ab}{\mathrm{Ab}}
\newcommand{\Grp}{\mathrm{Grp}}

\newcommand{\K}{{{\mathbf K}}}

\renewcommand{\H}{{{\mathbf H}}}

\newcommand{\Singaone}{\operatorname{Sing}^{\aone}\!\!}

\renewcommand{\setminus}{\smallsetminus}

\def\adj{\leftrightarrows}

\newcommand{\cof}{\mathrm{cof}}
\newcommand{\fib}{\mathrm{fib}}

\newcommand{\Addresses}{{
\bigskip
\footnotesize

A.~Asok, Department of Mathematics, University of Southern California, 3620 S.~Vermont Ave., Los Angeles, CA 90089-2532, United States; E-mail address: \url{asok@usc.edu}
\medskip

T.~Bachmann, Department of Mathematics, JGU Mainz, Staudingerweg 9, 55128 Mainz, Germany; E-mail address: \url{tom.bachmann@zoho.com}
\medskip

M.J.~Hopkins, Department of Mathematics, Harvard University, One Oxford Street, Cambridge, MA 02138, United States \textit{E-mail address:} \url{mjh@math.harvard.edu}
}}

\newcounter{intro}
\theoremstyle{plain}
\newtheorem{theorem}{Theorem}[section]

\newtheorem{lem}[theorem]{Lemma}
\newtheorem{cor}[theorem]{Corollary}
\newtheorem{proposition}[theorem]{Proposition}
\newtheorem*{claim*}{Claim} 

\newtheorem*{thm*}{Theorem}
\newtheorem*{problem*}{Problem}

\newtheorem{thmintro}{Theorem}

\theoremstyle{definition}
\newtheorem{defn}[theorem]{Definition}
\newtheorem{construction}[theorem]{Construction}

\theoremstyle{remark}
\newtheorem{rem}[theorem]{Remark}
\newtheorem{remintro}[thmintro]{Remark}

\newtheorem{ex}[theorem]{Example}

\newtheorem{entry}[theorem]{}

\numberwithin{equation}{subsection}

\begin{document}
\pagestyle{fancy}
\renewcommand{\sectionmark}[1]{\markright{\thesection\ #1}}
\fancyhead{}
\fancyhead[LO,R]{\bfseries\footnotesize\thepage}
\fancyhead[LE]{\bfseries\footnotesize\rightmark}
\fancyhead[RO]{\bfseries\footnotesize\rightmark}
\chead[]{}
\cfoot[]{}
\setlength{\headheight}{1cm}

\author{Aravind Asok\thanks{Asok was partially supported by National Science Foundation Awards DMS-1802060 and DMS-2101898} \and Tom Bachmann \and Michael J. Hopkins\thanks{Hopkins was partially supported by National Science Foundation Award DMS-1810917}}

\title{{\bf On the Whitehead theorem for nilpotent motivic spaces}}
\date{}
\maketitle

\begin{abstract}
	We improve some foundational connectivity results and the relative Hurewicz theorem in motivic homotopy theory, study functorial central series in motivic local group theory, establish the existence of functorial Moore--Postnikov factorizations for nilpotent morphisms of motivic spaces under a mild technical hypothesis and establish an analog of the Whitehead theorem for nilpotent motivic spaces.  As an application, we deduce a surprising unstable motivic periodicity result.
\end{abstract}

\begin{footnotesize}
\tableofcontents
\end{footnotesize}

\section{Introduction}
In classical algebraic topology, Whitehead proved that a map of simply connected spaces that induces an isomorphism on homology is a weak equivalence.  This result was generalized in various directions:  Dror introduced the notion of a nilpotent space and established that a map of nilpotent spaces that is a homology isomorphism is necessarily a weak equivalence \cite[Theorem 3.1]{DrorWhitehead}.  In modern terminology, these results establish that the stabilization functor (from unstable to stable homotopy theory) is conservative on the subcategory of nilpotent spaces.  

In this work, we analyze versions of these results in motivic homotopy theory.  Assume $k$ is a field and write $\Sm_k$ for the category of smooth $k$-schemes.  The $\infty$-category of motivic spaces $\ho{k}$ is a (left Bousfield) localization of the $\infty$-category of presheaves of spaces on $\Sm_k$. Following \cite[Definition 3.2.1]{AFHLocalization}, we say that a motivic space is {\em nilpotent} if its fundamental (Nisnevich) sheaf of groups is $\aone$-nilpotent and acts $\aone$-nilpotently on the higher homotopy sheaves (we called such spaces $\aone$-nilpotent in \cite{AFHLocalization}; our terminology differs slightly from that used in the previous paper and we recall in the notation and conventions below). 

In the context of motivic homotopy theory, nilpotent spaces arise more readily than $1$-connected spaces.  Indeed, the only $1$-connected smooth proper variety over a field is $\Spec k$ \cite[Proposition 5.1.4]{AM}.  On the other hand, $\pone$ or, more generally, flag varieties for split simply connected semi-simple algebraic groups are nilpotent \cite[Theorem 3.4.8]{AFHLocalization} and have non-abelian fundamental groups!  Likewise, affine Grassmannians of split, simply connected, semi-simple group schemes are nilpotent by \cite[Theorem 15]{Bachmanngrassmannian} in conjunction with \cite[Theorem 3.3.17, Example 3.4.1]{AFHLocalization}.

In this note, we improve our techniques for dealing with nilpotent motivic spaces from \cite{AFHLocalization}.  Before treating nilpotent motivic spaces in general, we further expand the arsenal of tools for dealing with sheaves of groups whose classifying spaces are motivic local: these are F. Morel's ``strongly $\aone$-invariant sheaves of groups.  One key structural deficit of the category of strongly $\aone$-invariant sheaves of groups is that given an epimorphism of sheaves of groups whose source is strongly $\aone$-invariant and whose target is $\aone$-invariant, it is unclear whether the target is automatically strongly $\aone$-invariant.  We axiomatize this property with the notion of a very strongly $\aone$-invariant sheaf of groups (see Definition~\ref{defn:verystronglya1invariant}).  We then show that the class of very strongly $\aone$-invariant sheaves of groups contains many examples of interest.  

We adapt various functorial constructions of classical group theory to the motivic setting.  We show that actions of $\aone$-nilpotent sheaves of groups have {\em functorial} central series (Proposition~\ref{prop:functorialpicentralseries}), we improve the relative Hurewicz theorem from \cite[Theorem 4.2.1]{AFHLocalization} (Theorem~\ref{thm:relativeHurewicz}), we show that nilpotent morphisms of motivic spaces have {\em functorial} principal refinements (Theorem \ref{thm:nilpotentprincipalrefinement}) under mild hypotheses, and we demonstrate the Whitehead theorem for nilpotent motivic spaces (see Theorem~\ref{thm:whitehead}).

In order to state our results, we use the $S^1$-stable motivic homotopy category $\SH^{S^1}(k)$ as described by F. Morel \cite[\S 4]{MStable}.  Recall that there is a stabilization functor $\Sigma^{\infty}_{S^1}: \ho{k} \to \SH^{S^1}(k)$ which is Kan extended from the functor sending $X \in \Sm_k$ to its suspension spectrum $\Sigma^{\infty}_{S^1}X_+$, where $X_+$ is $X$ with a disjoint base-point attached.  


\begin{thmintro}[See {Theorem~\ref{thm:whitehead}}]
	If $k$ is a perfect field and $f: \mathscr{X} \to \mathscr{Y}$ is a morphism of nilpotent motivic $k$-spaces such that $\Sigma^{\infty}_{S^1}f$ is an equivalence, then $f$ is an equivalence.  
\end{thmintro}

In the body of the text, we actually prove a version of this result with explicit connectivity bounds: loosely, if $n\geq 0$ is any integer, and $\Sigma^{\infty}_{S^1}f$ has $n$-connected fiber, then $f$ also has $n$-connected fiber.  In the classical setting, the corresponding result when $n = 1$ follows from a result of Stallings \cite[Theorem 2.1]{StallingsCS}.  Dror's proof of the Whitehead theorem generalizes Stallings' approach.  Actually, Dror's result established more than the Whitehead theorem for nilpotent spaces, and an even simpler proof was given slightly later by Gersten \cite{GerstenWhitehead}. 

Stallings' proof uses the Serre spectral sequences, but he suggests a way to avoid their use (though remarks the advantage of such an approach would be ``Pyrrhic").  For nilpotent spaces, it is possible to give a proof avoiding the Serre spectral sequence by combining connectivity estimates, the relative Hurewicz theorem and functorial Postnikov resolutions.  Unfortunately, the known versions of the Serre spectral sequence in motivic homotopy theory are not useful in our context and so we follow the more hands-on approach suggested in the previous paragraph.  We establish and appeal to a mild improvement of the Blakers--Massey theorem in motivic homotopy theory relying on a motivic analog of a result of Ganea \cite{Ganea}.  As such, our approach to the Whitehead theorem is undoubtedly related to the treatment in \cite{Toomer}.

We close with some applications of the Whitehead theorem.  In motivic homotopy theory, there are many benefits to being able to check whether a map is a weak equivalence stably.  In particular, working stably, one can make use of geometric constructions that allow singularity formation (e.g., by working in the $cdh$-topology) via a theorem of Voevodsky.  As a concrete application of this principle, we establish the following result.  

Write $\mathrm{HP}^{\infty}$ for the Panin--Walter \cite{PaninWalter} model for $BSL_2$. The ``inclusion of the bottom cell" $\mathrm{HP}^1 \hookrightarrow \mathrm{HP}^{\infty}$ (classifying the Hopf bundle $\nu$) yields a map $S^{4,2} \to \mathrm{HP}^{\infty}$, which is adjoint to a map $S^{3,1} \longrightarrow \Omega^{1,1} \op{HP}^{\infty}$.  The corresponding map in classical homotopy theory is evidently an equivalence, and the analog of this map in $C_2$-equivariant homotopy theory is an equivalence by \cite[Proposition 3.6]{HahnWilson}.  The next result, which we found surprising because of the ``weight shifting" phenomena it displays, asserts that this map is a motivic equivalence; the aforementioned ``classical" statements follow immediately by appeal to suitable realizations.  

\begin{thmintro}[See Theorem~\ref{thm:exoticequivalence}]
	\label{thmintro:exoticequivalence}
	If $k$ has characteristic $0$, then the map
	\[
	S^{3,1} \longrightarrow \Omega^{1,1} \op{HP}^{\infty}
	\]	
	deescribed above is an equivalence.  
\end{thmintro}

Here, the introduction of singularities arises by our appeal to affine Grassmannians.  The affine Grassmannian $\op{Gr}_{G}$ for $G$ a split, semi-simple, simply connected algebraic group is an increasing union of {\em singular} algebraic varieties.  Moreover, a motivic variant of a result due to Quillen \cite{MitchellQuillen} and Garland-Raghunathan \cite{GarlandRaghunathan} due to the second author \cite[Theorem 15]{Bachmanngrassmannian} implies that $\op{Gr}_{G}$ has the structure of an $h$-space.  The map from Theorem~\ref{thmintro:exoticequivalence} is obtained from a map of $h$-spaces after looping and we use the Whitehead theorem above to check that map is an equivalence.

In another direction, using the full strength of our Whitehead theorem, Proposition~\ref{prop:exceptionalEHP} allows to construct some EHP-sequences for motivic spheres outside the range considered by Wickelgren and Williams \cite{WickelgrenWilliams}.  We use this exceptional EHP sequence to obtain Corollary~\ref{cor:cohenexceptional}, which provides a motivic analog of F. Cohen's exceptional $2$-local fiber sequence \cite[Theorem 3.3]{Cohen} (which he attributes to H. Toda).



\subsubsection*{Notation, conventions and relation to other work}
Assume $k$ is a field; we will frequently assume $k$ is perfect.  Write $\Sm_k$ for the category of smooth $k$-schemes.  We write $\mathrm{P}(\Sm_k)$ for the $\infty$-category of presheaves of spaces on $\Sm_k$, and $\Shv_{\Nis}(\Sm_k) \subset \mathrm{P}(\Sm_k)$ for the full subcategory of Nisnevich sheaves of spaces on $\Sm_k$, which is an $\infty$-topos.  There are corresponding pointed versions of all these constructions: we write $\mathrm{P}(\Sm_k)_*$ for the $\infty$-category of presheaves of pointed spaces and $\Shv_{\Nis}(\Sm_k)_*$ for the subcategory of Nisnevich sheaves of pointed spaces.  We write $S^i$ for the simplicial $i$-circle, and we set $S^{p,q} := S^{p-q} \wedge \gm{\sma q}$.  We then write $\Sigma$ for $S^1 \wedge$ on $\mathrm{P}(\Sm_k)_*$.  We will write $\bpi_i^{\Nis}$ for the (Nisnevich) homotopy sheaves of any $\mathscr{X} \in \mathrm{P}(\Sm_k)_*$.
 
Suppose $f: \mathscr{X} \to \mathscr{Y}$ is a morphism in $\Shv_{\Nis}(\Sm_k)$.
We shall say that \emph{$f$ has $n$-connected fibers} if $\mathscr X$ is $n$-connected as an object of $\Shv_\Nis(\Sm_k)_{/\mathscr Y}$.  For example, if $\mathscr Y$ is connected, this is the same as requiring that $\fib(f) \in \Shv_\Nis(\Sm_k)$ be $n$-connected.
\tom{I removed all references to $n$-connected maps (I hope...)}

This paper may be viewed as a prequel to \cite{ABHFreudenthal} and the results established here will be used there.  As is perhaps clear from the discussion above, and in contrast to \cite{AFHLocalization}, this paper is written using the language of $\infty$-categories.  As such, terms like limit and colimit or fiber and cofiber will be used in the $\infty$-categorical sense {\em in contrast} to their usage in \cite{AFHLocalization}.  The notation and terminology here are chosen to be consistent with \cite{ABHFreudenthal}.

\begin{remintro}
	\label{rem:choudhuryhogadi}
	A previous version of this paper relied on \cite[Theorem 1.4]{HogadiChoudhury} asserting that a quotient of a strongly $\aone$-invariant sheaf of groups is $\aone$-invariant if and only if it is strongly $\aone$-invariant.  However, at the moment, the proof of \cite[Lemma 2.9]{HogadiChoudhury} contains a gap, which renders the proof of \cite[Theorem 1.4]{HogadiChoudhury} incomplete.  This note circumvents appeal to \cite{HogadiChoudhury} at the expense of some slight additional technical hypotheses which are satisfied in all cases where we will make use of them.
\end{remintro}

\section{Motivic localization and group theory}
\label{ss:motiviclocalization}
A presheaf $\mathscr{X} \in \mathrm{P}(\Sm_k)$ is $\aone$-invariant if for every $U \in \Sm_k$ the projection $U \times \aone \to U$ induces an equivalence $\mathscr{X}(U) \to \mathscr{X}(U \times \aone_k)$.  Similarly, $\mathscr{X}$ is said to be Nisnevich excisive if $\mathscr{X}(\emptyset)$ is contractible and $\mathscr{X}$ takes Nisnevich distinguished squares to cartesian squares.  The $\infty$-category of motivic spaces $\ho{k}$ is the full subcategory of $\mathrm{P}(\Sm_k)$ spanned by $\aone$-invariant and Nisnevich excisive spaces.  Since homotopy invariance and Nisnevich excision are defined by a small set of conditions, the inclusion $\ho{k} \subset \mathrm{P}(\Sm_k)$ is an accessible localization.  The category $\ho{k}$ is a presentable $\infty$-category and the inclusion $\ho{k} \subset \mathrm{P}(\Sm_k)$ admits a left adjoint
\[
\mathrm{L}_{mot}: \mathrm{P}(\Sm_k) \longrightarrow \ho{k}
\]
that we call the motivic localization functor.  

The functor $\mathrm{L}_{mot}$ is known to preserve finite products \cite[Proposition C.6]{Hoyois}.  We write $\bpi_i(\mathscr{X})$ for the homotopy sheaves of a motivic space (the notation $\bpi_i^{\aone}(-)$ is frequently used in previous work).  If $f$ is a map in $\ho{k}$, we write $\fib(f)$ and $\cof(f)$ for its fiber and cofiber.  The next result records the key property of the motivic localization functor.  

\begin{theorem}[Morel]
	\label{thm:unstableconnectivity}
	Suppose $n \geq -2$ is an integer and assume $k$ is a field and $\mathscr{X}$ is a space pulled back from a perfect subfield of $k$.  If $\mathscr{X}$ is $n$-connected, then $\mathrm{L}_{mot}\mathscr{X}$ is also $n$-connected.
\end{theorem}

\begin{proof}
	The statement is trivial for $n < 0$, and for $n = 0$ follows essentially from the construction of the motivic localization functor \cite[\S2 Corollary 3.22]{MV}.  For $n > 0$ it is \cite[Theorem 6.38]{MField} (though see \cite[Theorem 2.2.12]{AWW} for a detailed proof).
\end{proof}

Key to deducing the above results are the notions of strictly and strongly $\aone$-invariant sheaves of groups.  

\begin{defn}
A Nisnevich sheaf of groups $\mathbf{G}$ on $\Sm_k$ is {\em strongly $\aone$-invariant} if the cohomology presheaves $H^i(-,\mathbf{G})$ are $\aone$-invariant for $i = 0,1$.  A Nisnevich sheaf of abelian groups is {\em strictly $\aone$-invariant} if the (Nisnevich) cohomology presheaves $H^i(-,\mathbf{A})$ are $\aone$-invariant for all $i \geq 0$.  
\end{defn}

Write $\Grp^{\aone}_k$ for the category of strongly $\aone$-invariant sheaves of groups and $\Ab^{\aone}_k$ for the category of strictly $\aone$-invariant sheaves.  We use the following result without mention in the sequel.

\begin{proposition}
Assume $k$ is a field.  If $\mathbf{G}$ is a Nisnevich sheaf of groups on $\Sm_k$, then $\mathbf{G}$ is strongly $\aone$-invariant if and only if $B\mathbf{G}$ is motivic local.
\end{proposition}

\begin{proof}
	This statement is asserted without proof in \cite[Remark 1.8]{MField} and claims exist (e.g., on MathOverflow) that the two notions are not evidently equivalent because of base-point issues. For any Nisnevich sheaf of groups $\mathbf{G}$ we know that $B\mathbf{G}$ is $1$-truncated and $\pi_{1-i}(B\mathbf{G}(-)) \cong H^i(-,\mathbf{G})$ (see, for example \cite[Proposition 4.1.16]{MV}).  Thus, if $B\mathbf{G}$ is motivic local, then $\mathbf{G}$ is evidently strongly $\aone$-invariant.  Conversely, if $\mathbf{G}$ is strongly $\aone$-invariant, then we must show that $B\mathbf{G}(U) \to B\mathbf{G}(U \times \aone)$ is an equivalence for any smooth $k$-scheme $U$.  The remark above shows that the induced map on $\pi_0$ is a bijection.  It remains to prove that for any smooth $k$-scheme $U$, and any base-point $u \in \pi_0(B\mathbf{G}(U))$ the induced map $\pi_1(B\mathbf{G}(U),u) \to \pi_1(B\mathbf{G}(U \times \aone),u)$ (where we have abused notation for the base-point in $B\mathbf{G}(U \times \aone)$) is a bijection.  This map is once again a bijection when $u$ corresponds to the trivial torsor by the definition of strong $\aone$-invariance.  However, $\pi_1$ is a sheaf in the Nisnevich topology (it can be described explicitly as the automorphism group scheme of the $\mathbf{G}$-torsor corresponding to $u$; see, for example, \cite[Corollary 3.6]{Koizumi}), isomorphisms of sheaves can be checked locally.  Thus, by choosing a Nisnevich cover $V \to U$ trivializing $u$, we reduce to the case of the base-point corresponding to the trivial torsor, which we have already treated. 
\end{proof}

If $\mathscr{X}$ is a motivic space, F. Morel defined $\aone$-homology sheaves $\H_i^{\aone}(\mathscr{X})$ \cite[Definition 6.29]{MField}; these sheaves are strictly $\aone$-invariant.  If $\mathscr{X}$ is a pointed motivic space, then Morel showed the homotopy sheaves $\bpi_i(\mathscr{X})$ are strongly $\aone$-invariant for $i \geq 1$.  There is a forgetful functor from $\Ab^{\aone}_k$ to the category of strongly $\aone$-invariant sheaves of groups; this functor is fully faithful.  If $k$ is perfect, Morel showed \cite[Corollary 5.45]{MField} that strongly $\aone$-invariant sheaves of abelian groups are strictly $\aone$-invariant, i.e., the forgetful functor $\Ab^{\aone}_k$ to strongly $\aone$-invariant sheaves of abelian groups is an equivalence.


The category of strictly $\aone$-invariant sheaves is abelian \cite[Lemma 6.2.13]{MStable} as the heart of a $t$-structure on a suitable triangulated category.  In contrast, while the category $\Grp^{\aone}_k$ has colimits (e.g. cokernels) for formal reasons, computing them seems tricky.  To remedy this defect, we recall some ideas from \cite{BHS}.

\begin{defn}
For any $\mathscr{X} \in \mathrm{P}(\Sm_k)$, set $\mathbf{S}\mathscr{X} := \bpi_0^{\Nis}(\Singaone \mathscr{X})$ and define $\mathbf{S}^{\infty} = \colim_n \mathbf{S}^n$ where $\mathbf{S}^n$ is the $n$-fold iteration of $\mathbf{S}$.  
\end{defn}

\begin{theorem}
	\label{thm:sinfty}
	The following statements hold.
	\begin{enumerate}[noitemsep,topsep=1pt]
		\item The functor $\mathbf{S}^{\infty}$ preserves finite products.
		\item For any $\mathscr{X} \in \mathrm{P}(\Sm_k)$, the map $\mathscr{X} \to \mathbf{S}^{\infty}\mathscr{X}$ factors through an epimorphism $\bpi_0^{\Nis}(\mathscr{X}) \to \mathbf{S}^{\infty}\mathscr{X}$.
		\item The sheaf $\mathbf{S}^{\infty}\mathscr{X}$ is $\aone$-invariant and the map $\bpi_0^{\Nis}(\mathscr{X}) \to \mathbf{S}^{\infty}\mathscr{X}$ is the initial map from the former to an $\aone$-invariant Nisnevich sheaf.
	\end{enumerate}
\end{theorem}

\begin{proof}
	The first statement follows because $\bpi_0^{\Nis}$ and $\Singaone$ both preserve finite products.  The second statement is immediate because $\Singaone$ is defined as a colimit and colimits preserve epimorphisms.  The final statement is \cite[Theorem 2.13]{BHS}.
\end{proof}

\subsection*{Very strongly $\A^1$-invariant sheaves}
Let $\mathbf G \to \mathbf Q$ be an epimorphism of Nisnevich sheaves of groups, with $\mathbf G$ strongly $\A^1$-invariant and $\mathbf Q$ $\A^1$-invariant.
Ideally, $\mathbf Q$ would automatically be strongly $\A^1$-invariant (see Remark~\ref{rem:choudhuryhogadi} for further discussion of this point).  Unfortunately we do not know a proof of this fact in general, and we thus axiomatize the relevant property.

\begin{defn}
	\label{defn:verystronglya1invariant}
Let $\mathbf G$ be a Nisnevich sheaf of groups on $\Sm_k$.  We call $\mathbf G$ \emph{very strongly $\A^1$-invariant} if $\mathbf G$ is $\A^1$-invariant and every $\A^1$-invariant quotient of $\mathbf G$ is strongly $\A^1$-invariant.
\end{defn}

\begin{lem}
	\label{lem:a1invquotient}
	If $1 \to \mathbf{K} \to \mathbf{G} \to \mathbf{Q} \to 1$ is a short exact sequence of Nisnevich sheaves of groups with $\mathbf{G}$ strongly $\aone$-invariant, then $\mathbf{Q}$ is $\aone$-invariant if and only if $\mathbf{K}$ is strongly $\aone$-invariant. 
\end{lem}

\begin{proof}
	Note that $B\mathbf{G}$ is connected and motivic local by assumption.  Assume $\mathbf{Q}$ is $\aone$-invariant and consider the fiber sequence $\mathbf{Q} \to B\mathbf{K} \to B\mathbf{G}$.  In that case \cite[Lemma 2.2.10]{AWW} implies that $B\mathbf{K}$ is motivic local as well.  Conversely if $\mathbf{K}$ is strongly $\aone$-invariant, $\mathbf Q \simeq \fib(B\mathbf K \to B\mathbf G)$ is $\A^1$-invariant.
\end{proof}



We first show that the category of very strongly $\aone$-invariant sheaves of groups contains non-trivial objects.

\begin{proposition}
	\label{prop:strictimpliesverystrong}
 If $k$ is a perfect field and $\mathbf{A}$ is abelian and strongly $\aone$-invariant, then $\mathbf{A}$ is very strongly $\aone$-invariant.
\end{proposition}

\begin{proof}
	Assume $\mathbf{A}$ is abelian and strongly $\aone$-invariant; we conclude that $\mathbf{A}$ is strictly $\aone$-invariant \cite[Corollary 5.45]{MField}.  Let $\mathbf{A}''$ be an $\aone$-invariant quotient of $\mathbf{A}$.  In that case, Lemma~\ref{lem:a1invquotient} implies that the kernel $\mathbf{A}'$ of the homomorphism $\mathbf{A} \to \mathbf{A}''$ is strongly $\aone$-invariant and abelian, whence once again strictly $\aone$-invariant.  In that case, $\mathbf{A}' \subset \mathbf{A}$ is a central subsheaf, and the fact that $\mathbf{A}''$ is strongly $\aone$-invariant follows from \cite[Lemma 3.1.4(3)]{AFHLocalization}.
\end{proof}


Next, we analyze permanence properties of the notion of very strongly $\aone$-invariant sheaves of groups. 

\begin{lem} 
	\label{lemm:very-strongly-permanence}
The subcategory of $\Grp^{\aone}_k$ consisting of very strongly $\aone$-invariant sheaves of groups is closed under formation of 
\begin{enumerate}[noitemsep,topsep=1pt]
	\item quotients, 
	\item extensions, and 
	\item filtered colimits.
\end{enumerate}  
\todo{closure under subobjects seems not clear}
\end{lem}

\begin{proof}
For closure under quotients, observe that if $\mathbf{G}$ is very strongly $\aone$-invariant and $\mathbf{Q}$ is a quotient of $\mathbf{G}$, then any quotient $\mathbf{Q}'$ of $\mathbf{Q}$ is again a quotient of $\mathbf{G}$.  In particular if $\mathbf{Q}'$ is $\aone$-invariant, it is automatically strongly $\aone$-invariant.

For closure under extensions, suppose we have a short exact sequence of the form
\[ 
1 \longrightarrow \mathbf K \longrightarrow \mathbf G \longrightarrow \mathbf Q \to 1. 
\]
If $\mathbf{K}$ and $\mathbf{Q}$ are strongly $\aone$-invariant, it follows that $\mathbf{G}$ is strongly $\aone$-invariant as well by appeal to Lemma~\ref{lem:a1invquotient}.  Suppose $\mathbf{G} \to \mathbf{G}'$ is an epimorphism with $\mathbf{G}'$ an $\aone$-invariant sheaf of groups.  Write $\mathbf{K}'$ for the image of $\mathbf{K}$ in $\mathbf{G}'$.  Since $\mathbf{G}'$ is $\aone$-invariant, it follows that $\mathbf{K}'$ is $\aone$-invariant as well, whence strongly $\aone$-invariant by the very strong $\aone$-invariance of $\mathbf{K}$.  Since $\mathbf{K}$ is a normal subgroup sheaf of $\mathbf{G}$, it also follows that $\mathbf{K}'$ is a normal subgroup sheaf of $\mathbf{G}'$.  In that case, write $\mathbf{Q}'$ for the quotient of $\mathbf{G}'$ by $\mathbf{K}'$.  The fiber sequence
\[
\mathbf{G}' \longrightarrow \mathbf{Q}' \longrightarrow B\mathbf{K}'
\]
in conjunction with the fact that $B\mathbf{K}'$ is motivic local shows that $\mathbf{Q}'$ is $\aone$-invariant by appeal to \cite[Lemma 2.2.10]{AWW}.  Since $\mathbf{Q}'$ is also a quotient of $\mathbf{Q}$ it is necessarily strongly $\aone$-invariant.  It follows that $\mathbf{G}'$ is an extension of strongly $\aone$-invariant sheaves and therefore itself strongly $\aone$-invariant by appeal to \cite[Lemma 3.1.14(2)]{AFHLocalization}, as required.
 	


Finally, to establish stability under filtered colimits, suppose $\mathbf{G}: \mathbf{I} \to \Grp^{\aone}_k$ is a filtered diagram of strongly $\aone$-invariant sheaves of groups.  Consider an $\aone$-invariant quotient $\colim_{\mathbf{I}} \mathbf{G} \to \mathbf{Q}$.  Write $\mathbf{Q}(i)$ for the image of $\mathbf{G}(i)$ in $\mathbf{Q}$.  In that case, $\colim_{i \in I} \mathbf{Q}(i) \weq \mathbf{Q}$.  Each $\mathbf{Q}(i)$ is $\aone$-invariant being a subsheaf of $\mathbf{Q}$.  Therefore, each $\mathbf{Q}(i)$ is strongly $\aone$-invariant as a quotient of the very strongly $\aone$-invariant sheaf $\mathbf{G}(i)$.  Since motivic local spaces are stable under filtered colimits, we conclude that $B\mathbf{Q} = \colim_i B\mathbf{Q}(i)$ is motivic local as well and conclude. \NB{I don't know a reference where this is explicitly stated.} 
\end{proof}

\begin{defn}
	\label{defn:aonesolvable}
	A strongly $\aone$-invariant sheaf of groups $\mathbf{G}$ is called {\em $\aone$-solvable} if there exists a finite increasing filtration:
	\[
	1 = \mathbf{G}_0 \subset \mathbf{G}_1 \subset \cdots \subset \mathbf{G}_n = \mathbf{G}
	\]	
	such that, for every $i$, $\mathbf{G}_i$ is strongly $\aone$-invariant, $\mathbf{G}_i$ is a normal subgroup sheaf of $\mathbf{G}_{i+1}$, and the successive subquotients $\mathbf{G}_{i+1}/\mathbf{G}_i$ are strongly $\aone$-invariant sheaves of abelian groups.
\end{defn}

We now give a rather general class of very strongly $\aone$-invariant sheaves of groups.  Recall that a Nisnevich sheaf of groups $\mathbf{G}$ is locally nilpotent if it admits a {\em descending} central series that is stalkwise finite \cite[Definition 2.1.3]{AFHLocalization} and nilpotent if it admits a finite descending central series. 

\begin{proposition}
	\label{prop:lnilpstrongimpliesverystrong}
Assume $k$ is a perfect field.  
\begin{enumerate}[noitemsep,topsep=1pt]
\item If $\mathbf{G}$ is $\aone$-solvable, then $\mathbf{G}$ is very strongly $\aone$-invariant.
\item If $\mathbf G$ is locally nilpotent and strongly $\A^1$-invariant, then $\mathbf{G}$ is very strongly $\aone$-invariant. \tom{Note that locally nilpotent + strongly $\A^1$-invariant does not seem to imply locally $\A^1$-nilpotent...}
\end{enumerate}
\end{proposition}

\begin{proof}
The first statement is an immediate consequence of the definition of $\aone$-solvable: use Proposition~\ref{prop:strictimpliesverystrong} in conjunction with Lemma~\ref{lemm:very-strongly-permanence} together with a straightforward induction argument on the length of an $\aone$-subnormal series.	
	
For the second statement, write $\mathbf Z_i \subset \mathbf G$ for the $i$-th higher center (see \cite[p. 675]{AFHLocalization}) of $\mathbf{G}$, then each $\mathbf Z_i$ is strongly $\A^1$-invariant \cite[Proposition 3.1.22]{AFHLocalization}.  On the other hand, since $\mathbf{G}$ is locally nilpotent, we conclude that $\mathbf G \simeq \colim_i \mathbf Z_i$ since this is true stalkwise by assumption.  To establish that $\mathbf{G}$ is very strongly $\aone$-invariant, it therefore suffices by appeal to Lemma~\ref{lemm:very-strongly-permanence}(3) to show that each $\mathbf{Z}_i$ is very strongly $\aone$-invariant.

By definition there are short exact sequences 
\[ 
1 \longrightarrow \mathbf Z_i \longrightarrow \mathbf Z_{i+1} \longrightarrow \mathbf A_{i+1} \to 1, 
\] 
where $A_{i+1}$ is strongly $\A^1$-invariant and abelian (indeed $\mathbf A_{i+1} \simeq \ker(\mathbf G/\mathbf Z_i \to \mathbf G/\mathbf Z_{i+1})$ which is by definition $Z(\mathbf G/\mathbf Z_i)$, and thus is strongly $\A^1$-invariant by \cite[Proposition 3.1.22]{AFHLocalization} again).  Since $\mathbf{Z}_1$ coincides with the center of $\mathbf{G}$, it and each $\mathbf{A}_i$ are strongly $\aone$-invariant and abelian, and thus very strongly $\aone$-invariant by appeal to Proposition~\ref{prop:strictimpliesverystrong}.  By induction, it then follows from Lemma~\ref{lemm:very-strongly-permanence}(2) that $\mathbf{Z}_{i+1}$ is very strongly $\aone$-invariant for each $i \geq 2$.
\end{proof}

\begin{rem} \label{rmk:extremely-strongly-aone-inv}
It is not clear that a strongly $\aone$-invariant subsheaf of an $\aone$-solvable sheaf of groups is again $\aone$-solvable.  In contrast, note that a subsheaf of groups of a locally nilpotent sheaf of groups is again locally nilpotent.  It follows that any strongly $\aone$-invariant subsheaf of a locally nilpotent, strongly $\aone$-invariant sheaf of groups is automatically very strongly $\aone$-invariant.  We resist the temptation to call such sheaves of groups \emph{extremely} strongly $\aone$-invariant. \TODO{But did we really?}
\end{rem}

\begin{ex}
In \cite[Definition 3.2.1]{AFHLocalization}, we also introduced the notion of a (locally) $\aone$-nilpotent sheaf of groups: this is a strongly $\aone$-invariant sheaf of groups that admits a (locally finite) descending central series where all subquotients are strictly $\aone$-invariant.  Locally $\aone$-nilpotent sheaves of groups are automatically very strongly $\aone$-invariant.  While we showed that nilpotent strongly $\aone$-invariant sheaves of groups are $\aone$-nilpotent \cite[Proposition 3.2.3]{AFHLocalization}, we do not know whether {\em locally} nilpotent strongly $\aone$-invariant sheaves of groups are automatically locally $\aone$-nilpotent because it is unclear whether a locally finite ascending $\aone$-central series gives rise to a locally finite descending $\aone$-central series.  
\end{ex}

\begin{proposition}
	\label{prop:hogadichoudhury-RENAMED}
	Suppose $k$ is a field, and $\mathbf{G} \in \Grp^{\aone}_k$ is a very strongly $\aone$-invariant sheaf of groups.  If $\mathbf{H}$ is a quotient of $\mathbf{G}$, then $\mathbf{S}^{\infty}\mathbf{H}$ is the initial strongly $\aone$-invariant quotient of $\mathbf{G}$ admitting a map from $\mathbf{H}$.
\end{proposition}

\begin{proof}
The initial $\A^1$-invariant and strongly $\A^1$-invariant quotient of $\mathbf H$ coincide, since $\mathbf G$ is assumed very strongly $\A^1$-invariant.
The statement thus follows from Theorem~\ref{thm:sinfty}. 
\end{proof}



\section{Connectivity of fibers and cofibers}
\label{ss:connectivityoffibersandcofibers}
We now review various facts about comparison of fibers and cofibers, including connectivity estimates.

\subsubsection*{Comparing fibers and cofibers}
We record the following result about comparison of horizontal and vertical fibers in a commutative diagram for lack of a convenient reference.  It will be used repeatedly in the sequel; the statement is undoubtedly very old (e.g., see \cite[Lemma 2.1]{CMN} where it stated as ``well-known") and holds in any $\infty$-category with finite limits.  The consequence at the end of the statement about fibers of composites was observed by Quillen \cite[3.10 Remark]{QuillenHA}.

\begin{proposition}
	\label{prop:fibersofcomposites}	
	A commutative square in $\ho{k}$ of the form
	\[
	\xymatrix{
		\mathscr{X}_{00} \ar[r]^{g_0}\ar[d]^{f_0} & \mathscr{X}_{10} \ar[d]^{f_1} \\
		\mathscr{X}_{01} \ar[r]^{g_1} & \mathscr{X}_{11}
	}
	\]
	can be embedded in a commutative diagram of the form:
	\[
	\xymatrix{
		\mathscr{X} \ar[d]\ar[r] & \fib(f_0) \ar[d]\ar[r] & \fib(f_1) \ar[d] \\
		\fib(g_0) \ar[d]\ar[r] & \mathscr{X}_{00} \ar[r]^{g_0}\ar[d]^{f_0} & \mathscr{X}_{10} \ar[d]^{f_1} \\
		\fib(g_1) \ar[r] & \mathscr{X}_{01} \ar[r]^{g_1} & \mathscr{X}_{11}
	}
	\]
	where (a) there is an equivalence $\mathscr{X} \cong \fib(\mathscr{X}_{00} \to \mathscr{X}_{01} \times_{\mathscr{X}_{11}} \mathscr{X}_{10})$, and (b) all rows and columns are fiber sequences.  In particular, if $f: \mathscr{X} \to \mathscr{Y}$ and $g: \mathscr{Y} \to \mathscr{Z}$ are maps in $\ho{k}$, then there is a fiber sequence of the form
	\[
	\fib(f) \longrightarrow \fib(g \circ f) \longrightarrow \fib(g).
	\]
\end{proposition}

\begin{proof}
	The second statement is a special case of the first arising from comparison of horizontal and vertical homotopy fibers in the commutative diagram:
	\[
	\xymatrix{\mathscr{X} \ar[r]^{f}\ar[d]^{g \circ f} & \mathscr{Y} \ar[d]^{g} \\
		\mathscr{Z} \ar[r]^{id} & \mathscr{Z}.	
	}
	\]
\end{proof}	

\begin{rem}
	Given a commutative square as in Proposition~\ref{prop:fibersofcomposites}, there is, of course, a dual statement comparing vertical and horizontal cofibers.  The statement about cofibers of composites in the context of model categories can be found in \cite[Proposition 6.3.6]{Hovey}.
\end{rem}

\subsubsection*{Connectivity of fibers and cofibers}
\begin{proposition}
	\label{prop:cofiberconnectivity}
	Suppose $k$ is a field.  Assume $f: \mathscr{X} \to \mathscr{Y}$ is a morphism of pointed spaces in $\ho{k}$ that is pulled back from a perfect subfield of $k$. 
	\begin{enumerate}[noitemsep,topsep=1pt]
		\item If $f$ has $(n-1)$-connected fibers, then $\cof(f)$ is $n$-connected.
		\item If $\mathscr{X}$ and $\mathscr{Y}$ are $1$-connected, and $\cof(f)$ is $n$-connected, then $\fib(f)$ is $(n-1)$-connected.
	\end{enumerate}
\end{proposition}

\begin{proof}
	For the first statement, we proceed as follows.  Write $\mathscr{C}$ for the cofiber of $f$ computed in $\Shv_{\Nis}(\Sm_k)$, which fits in a commutative square of the form
	\[
	\xymatrix{
		\mathscr{X} \ar[r]^{f}\ar[d]& \mathscr{Y} \ar[d] \\
		\ast \ar[r] & \mathscr{C}. 
	}
	\]  
	By definition, $\cof(f) = \mathrm{L}_{mot}\mathscr{C}$.  Since the morphism $\ast \to \mathscr{C}$ is the cobase change of $f$ along $\mathscr{X} \to \ast$ and since morphisms with $(n-1)$-connected fibers are stable under cobase change along arbitrary morphisms \cite[Corollary 6.5.1.17]{HTT}, it follows that $\ast \to \mathscr{C}$ has $(n-1)$-connected fibers, i.e., $\mathscr{C}$ is $n$-connected.  The fact that $\cof(f)$ is $n$-connected then follows by appeal to Theorem~\ref{thm:unstableconnectivity}.
	
	For the second statement, assume that $\cof(f)$ is $n$-connected.  Set $\mathscr{F} := \fib(f)$.  Since $f$ is a map of $1$-connected spaces, it follows that $\mathscr{F}$ is automatically connected and $\bpi_1(\mathscr{F})$ is abelian.  Consider the comparison map $\mathscr{F} \to \Omega \mathscr{C}$ induced by taking fibers in the diagram from the preceding paragraph.  The classical Blakers--Massey theorem \cite[Theorems 3.10-11]{GoerssJardine} implies that if $\mathscr{F}$ is $m$-connected for some integer $m$ (which we may assume without loss of generality to be $\geq 0$), then $\mathscr{F} \to \Omega \mathscr{C}$ induces an isomorphism on homotopy sheaves in degrees $\leq m+1$.  We conclude that $\bpi_{i+1}(\Omega \mathscr{C})$ is strictly $\aone$-invariant for any $m \ge i \geq 0$.  In that case, the map $\bpi_{i+1}(\Omega \mathscr{C}) \to \bpi_{i+1}(\cof(f))$ is an isomorphism by \cite[Theorem 2.3.8]{AWW}.  Applying this observation iteratively, we deduce that the statement holds as long as $m \leq n-1$, which is what we wanted to show.
\end{proof}

We now recall some consequences of ``homotopy distributivity', i.e., commutativity of certain limits and colimits that will be useful in the sequel.  The primordial form of homotopy distributivity we use is that colimits are universal in $\ho{k}$; see \cite[Proposition 3.15]{HoyoisEquiv} for this statement.  For topological spaces, the results we describe go back to work of V. Puppe \cite{Puppe} and M. Mather \cite{Mather}.  They were studied in a model-categorical framework in unpublished work of C. Rezk \cite{Rezk}.  The results were later analyzed in the context of motivic homotopy theory in the thesis of M. Wendt \cite{WendtI} and the translation to the context of $\infty$-topoi is straightforward.  The first point in the next statement originates from a result of Ganea \cite[Theorem 1.1]{Ganea}.   

\begin{lem}
\label{lem:ganea}
Suppose $k$ is a field, and $\fib(f) \stackrel{\iota}{\to} \mathscr{E} \stackrel{f}{\to} \mathscr{B}$ is a fiber sequence in $\ho{k}$.
\begin{enumerate}[noitemsep,topsep=1pt]
\item There is a natural equivalence
\[
\fib(\cof(\iota) \to \mathscr{B}) \cong \Sigma \Omega \mathscr{B} \wedge \fib(f).
\]
\item There is a cofiber sequence of the form
\[
\Sigma \fib(f) \longrightarrow \cof(f) \longrightarrow \cof(\cof(\iota) \to \mathscr{B})
\]
\end{enumerate}
\end{lem}

\begin{proof}
The first statement may be found in \cite[Lemma 2.27]{DH}.  For the second point, consider the commutative diagram 
\[
\xymatrix{
	\mathscr{E} \ar@{=}[r]\ar[d] & \mathscr{E} \ar[d]^{f} \\
	\cof(\iota) \ar[r] & \mathscr{B}. 
}
\]
Comparing vertical and horizontal cofibers via the dual of Proposition~\ref{prop:fibersofcomposites}, one obtains a cofiber sequence of the form:
\[
\cof(\mathscr{E} \to \cof(\iota)) \longrightarrow \cof(f) \longrightarrow \cof(\cof(\iota) \to \mathscr{B}).
\]
On the other hand, another application of the dual of Proposition~\ref{prop:fibersofcomposites} to the commutative diagram 
\[
\xymatrix{
	\fib(f) \ar[r]^{\iota}\ar[d] & \mathscr{E} \ar[d] \\
	\ast \ar[r] & \cof(\iota)
}
\]
yields a cofiber sequence of the form
\[
\Sigma \fib(f) \longrightarrow \cof(\mathscr{E} \to \cof(\iota)) \longrightarrow \ast.
\]
Using the two diagrams above, and taking the cofiber of the composite $\Sigma \fib(f) \to \cof(\mathscr{E} \to \cof(\iota)) \to \cof(f)$, we obtain an equivalence
\[
\cof(\Sigma \fib(f) \to \cof(f)) \longrightarrow \cof(\cof(\iota) \to \mathscr{B}),
\]
which is precisely what we wanted to show.
\end{proof}

The next result improves the motivic analog of the Blakers--Massey theorem of \cite[Theorem 4.1]{AFComparison} or \cite[Theorem 2.3.8]{Strunk}.

\begin{proposition}[Homotopy excision]
\label{prop:blakersmassey}
Assume $k$ is a field, and $f: \mathscr{E} \to \mathscr{B}$ is a morphism pulled back from a perfect subfield of $k$.  If $m,n \geq 0$ are integers, the following statements regarding the fiber sequences $\fib(f) \stackrel{\iota}{\to} \mathscr{E} \stackrel{f}{\to} \mathscr{B}$ hold.
\begin{enumerate}[noitemsep,topsep=1pt]
	\item If $\fib(f)$ is $m$-connected and $\mathscr{B}$ is $n$-connected, then $\cof(\iota) \to \mathscr{B}$ has $(m+n+1)$-connected fibers.
	\item If $\fib(f)$ is $m$-connected and $\mathscr{B}$ is $n$-connected, then the canonical map $\Sigma \fib(f) \to \cof(f)$ has $(m+n+1)$-connected fibers.
\end{enumerate}
\end{proposition}

\begin{proof}
For the first statement, observe that Ganea's lemma \ref{lem:ganea} gives an equivalence $\Sigma\Omega \mathscr{B} \wedge \fib(f) \cong \fib(\cof(\iota) \to \mathscr{B})$.  Since $\mathscr{B}$ is $n$-connected, $\Sigma \Omega \mathscr{B}$ is at least $n$-connected.  Since $\fib(f)$ is $m$-connected, we conclude that $\Sigma \Omega \mathscr{B} \wedge \fib(f)$ is at least $(m+n+1)$-connected by appeal to \cite[Lemma 3.3.1]{AWW}.  Thus, $\cof(\iota) \to \mathscr{B}$ has $(m+n+1)$-connected fibers by definition.

For the second point, note that Lemma~\ref{lem:ganea}(2) yields a cofiber sequence of the form
\[
\Sigma \fib(f) \longrightarrow \cof(\iota) \longrightarrow \cof(\cof(\iota) \to \mathscr{B}).
\]
By the conclusion of Point (1), we see that $\cof(\iota) \to \mathscr{B}$ has $(m+n+1)$-connected fibers, and then Proposition~\ref{prop:cofiberconnectivity}(1) implies that $\cof(\cof(\iota) \to \mathscr{B})$ is $(m+n+2)$-connected.

The map $f$ has $m$-connected fibers for some $m \geq 0$ by assumption, so another application of Proposition~\ref{prop:cofiberconnectivity}(1)
implies	that $\cof(f)$ is at least $m+1$-connected, in particular simply connected.  Likewise, $\Sigma \fib(f)$ is at least $m+1$-connected and therefore also at least simply connected.  It follows that $\Sigma \fib(f) \to \cof(f)$ is a map of $1$-connected spaces whose cofiber is $m+n+2$-connected.  Thus, appeal to Proposition~\ref{prop:cofiberconnectivity}(2) implies that $\Sigma \fib(f) \to \cof(f)$ has $(m+n+1)$-connected fibers as well.
\end{proof}

We may use the above results to strengthen the relative Hurewicz theorem of \cite[Theorem 4.2.1]{AFHLocalization}, but we give a self-contained treatment here.  Suppose $f: \mathscr{E} \to \mathscr{B}$ is a morphism of pointed, connected spaces.  There is a morphism of fiber sequences of the form
\[
\xymatrix{
\fib(f) \ar[r]\ar[d] & \mathscr{E} \ar[r]^{f}\ar[d] & \mathscr{B} \ar[d]\\
\Omega \cof(f) \ar[r]& \ast \ar[r]& \cof(f). 
}
\]
The diagram shows that there is an induced action of $\bpi_1(\mathscr{E})$ on $\Omega \cof(f)$, which is necessarily trivial.  The relative Hurewicz theorem analyzes the map $\fib(f) \to \Omega \cof(f)$ or, rather, the induced map $\Sigma \fib(f) \to \cof(f)$.

\begin{theorem}
	\label{thm:relativeHurewicz}
	Suppose $k$ is a field, and assume $f: \mathscr{E} \to \mathscr{B}$ is a morphism of pointed connected spaces that is pulled back from a perfect subfield of $k$.  If $f$ has $(n-1)$-connected fibers, then the following statements hold.
	\begin{enumerate}[noitemsep,topsep=1pt]
		\item The space $\cof(f)$ is $n$-connected.
		\item The relative Hurewicz map $\bpi_n(\fib(f)) \to \H_{n+1}^{\aone}(\cof(f))$ is the initial morphism from $\bpi_n(\fib(f))$ to a strictly $\aone$-invariant sheaf on which $\bpi_1(\mathscr{E})$ acts trivially.
	\end{enumerate}
\end{theorem}

\begin{proof}
	The first statement is Proposition~\ref{prop:cofiberconnectivity}(1); thus we prove the second.  We treat only the case $n = 1$, as the case $n \geq 2$ follows the same argument but is easier.  The general principle is as follows.  If $\mathscr{X}$ is an $(n-1)$-connected Nisnevich sheaf of spaces, then $\bpi_n(\mathrm{L}_{mot}\mathscr{X})$ is the initial strongly $\aone$-invariant sheaf under $\bpi_n(\mathscr{X})$. By the classical relative Hurewicz theorem (in $\Shv_{\Nis}(\Sm_k)$) we know that $\bpi_1(\fib(f)) \to \H_{2}(\mathscr{C})$ is the initial map from $\bpi_1(\fib(f))$ to a sheaf on which $\bpi_1(\mathscr{E})$ acts trivially. Since $\H_2^{\aone}(\cof(f)) \simeq \bpi_2(\cof(f))$ is the initial strictly $\aone$-invariant sheaf under $\H_{2}(\mathscr{C})$, the result follows.

\end{proof}

For later use, we record the following version of Morel's suspension theorem \cite[Theorem 6.61]{MField}.

\begin{theorem}
	\label{thm:refinedsimplicialsuspension}
	Assume $n \geq 0$ is an integer and $\mathscr{X}$ is a pointed, $n$-connected motivic space.
	\begin{enumerate}[noitemsep,topsep=1pt] 
		\item For any integer $i \geq 0$, the canonical map $\mathscr{X} \to \Omega^i\Sigma^i \mathscr{X}$ has $2n$-connected fibers.
		\item The map $\mathscr{X} \to \Omega^{\infty}_{S^1}\Sigma^{\infty}_{S^1}\mathscr{X}$ has $2n$-connected fibers.
	\end{enumerate}
\end{theorem}

\begin{proof}
	For $i = 0$ there is nothing to show, so we assume $i \geq 1$.  In that case, Theorem~\ref{thm:unstableconnectivity} implies that $\Sigma^i\mathscr{X}$ is at least $n+i$-connected and thus that $\Omega^i\Sigma^i \mathscr{X}$ is at least $n$-connected as well.  Theorem~\ref{thm:relativeHurewicz} implies that $\bpi_{n+1}(\Omega^i\Sigma^i \mathscr{X})$ coincides with $\H_{n+1}^{\aone}(\mathscr{X})$.  With that in mind, when $n = 0$, the conclusion of the theorem is that the map $\bpi_1(\mathscr{X}) \to \H_1^{\aone}(\mathscr{X})$ is an epimorphism, which follows immediately from Theorem~\ref{thm:relativeHurewicz} (this particular case also is contained in \cite[Theorem 1.4]{HogadiChoudhury}).  
	
	We now assume that $n > 0$.  In that case, write $\mathscr{X'}$ for $k$-fold suspension of $\mathscr{X}$ in $\Shv_{\Nis}(\Sm_k)$.  The map $\mathscr{X} \to \Omega^i \mathscr{X}'$ in $\Shv_{\Nis}(\Sm_k)$ has $2n$-connected fibers by the classical Freudenthal suspension theorem applied stalkwise.  In other words, there is a fiber sequence in $\Shv_{\Nis}(\Sm_k)$ of the form
	\[
	\mathscr{F} \longrightarrow \mathscr{X} \longrightarrow \Omega^i\mathscr{X}'
	\]     
	with $\mathscr{F}$ $2n$-connected.  In particular, we conclude that $\bpi_j(\Omega^i \mathscr{X}')$ is strongly $\aone$-invariant for $j \leq 2n$.  Since $n \geq 1$, we conclude from \cite[Theorem 2.3.3]{AWW} that the canonical map $\mathrm{L}_{mot}\mathscr{F} \to \mathrm{L}_{mot}(\fib(\mathscr{X} \to \Omega^i \mathscr{X}'))$ is an equivalence.  On the other hand, repeated application of \cite[Theorem 2.4.1]{AWW} implies that $\mathrm{L}_{mot} \Omega^i \mathscr{X}' \to \Omega^i \mathrm{L}_{mot}\mathscr{X}'$ is an equivalence as well.  Combining the observations above, there is thus a fiber sequence of the form
	\[
	\mathrm{L}_{mot}\mathscr{F} \longrightarrow \mathscr{X} \longrightarrow \Omega^i\Sigma^i \mathscr{X}.
	\]
	The unstable connectivity theorem implies $\mathrm{L}_{mot}\mathscr{F}$ is $2n$-connected, which concludes the verification of Point(1).  Point (2) follows immediately from Point (1) by passage to the limit.
\end{proof}

\begin{ex}
	 Assume $\mathbf{G}$ is a connected group (e.g., $SL_2$).  In that case, the identity map on $\mathbf{G}$ factors through a map $\mathbf{G} \to \Omega \Sigma \mathbf{G} \to \mathbf{G}$, which shows that $\mathbf{G}$ is a retract of $\Omega \Sigma \mathbf{G}$.  The other factor can be described explicitly as follows.  Applying Lemma~\ref{lem:ganea} to the fiber sequence $\mathbf{G} \to \ast \to B\mathbf{G}$, one concludes the existence of a fiber sequence of the form
	\[
	\Sigma \mathbf{G} \wedge \mathbf{G} \longrightarrow \Sigma \mathbf{G} \longrightarrow B\mathbf{G},
	\]
	where the map $\Sigma \mathbf{G} \wedge \mathbf{G} \to \Sigma \mathbf{G}$ is the Hopf construction of the multiplication on $\mathbf{G}$ \cite[p. 191]{MField}.  The retraction map described in the first paragraph provides a section of the corresponding fibration after looping and we conclude that 
	\[
	\Omega \Sigma \mathbf{G} \cong \Omega \Sigma (\mathbf{G} \wedge \mathbf{G}) \times \mathbf{G}
	\]
	It follows that the fiber of $\mathbf{G} \to \Omega \Sigma \mathbf{G}$ can be identified with $\Omega^2 \Sigma \mathbf{G} \wedge \mathbf{G}$.  Since $\mathbf{G}$ is a connected group, the map $\bpi_1(\mathbf{G}) \to \bpi_1(\Omega \Sigma \mathbf{G})$ is an isomorphism by the Hurewicz theorem \ref{thm:relativeHurewicz}.  On the other hand $\bpi_1(\Omega^2\Sigma \mathbf{G} \wedge \mathbf{G})$ is non-trivial in general (e.g., for $SL_2$).  One concludes that connectivity assertion in  Theorem~\ref{thm:refinedsimplicialsuspension} is optimal for $n = 0$.  For $n \geq 1$, we refer the reader to \cite[Theorem 3.2.1]{AWW}.
\end{ex}

\section{Abelianization and $\aone$-lower central series}
Theorem~\ref{thm:sinfty} and Proposition~\ref{prop:hogadichoudhury-RENAMED} allow us to build appropriate analogs of functorial constructions from classical group theory.  To illustrate this, we discuss abelianization and lower central series for strongly $\aone$-invariant sheaves of groups.  Unfortunately, all of these notions are only provably well-behaved for very strongly $\aone$-invariant sheaves of groups in the sense of Definition~\ref{defn:verystronglya1invariant}.

\begin{lem}
	\label{lem:abelianization}
	Assuming $k$ a field, the following statements hold.
	\begin{enumerate}[noitemsep,topsep=1pt]
		\item The inclusion functor $\Ab^{\aone}_k \hookrightarrow \Grp^{\aone}_k$ admits a left adjoint $(-)^{ab}_{\aone}$ (which is thus right exact).
		\item If $k$ is furthermore perfect and $\mathbf{G}$ is a very strongly $\aone$-invariant sheaf, then the unit of the adjunction $\mathbf{G} \to \mathbf{G}^{ab}_{\aone}$ is an epimorphism.
		In fact $\mathbf{G}^{ab}_{\aone} \simeq \mathbf{S}^{\infty} \mathbf{G}^{ab}$.
	\end{enumerate}
\end{lem}

\begin{proof}
	The existence of an adjoint is a consequence of Morel's stable connectivity theorem.  We establish the second point by analyzing universal properties; to this end, write $\Grp_k$ for the category of Nisnevich sheaves of groups on $\Sm_k$, and $\Ab_k$ for the category of Nisnevich sheaves of abelian groups on $\Sm_k$.  Suppose $\mathbf{G}$ is a very strongly $\aone$-invariant sheaf of groups.  
	
	If $\mathbf{A}$ is a strongly $\aone$-invariant abelian sheaf of groups, then writing $\mathbf{G}^{ab}$ for abelianization as a Nisnevich sheaf of groups, there is a bijection
	\[
	\hom_{\Grp^{\aone}_k}(\mathbf{G},\mathbf{A}) \cong \hom_{\Ab_k}(\mathbf{G}^{ab},\mathbf{A}),
	\]  
	by the universal property of abelianization.  On the other hand, since $\mathbf{A}$ is $\aone$-invariant, we deduce from Theorem~\ref{thm:sinfty} that there is a bijection:
	\[
	\hom_{\Ab_k}(\mathbf{G}^{ab},\mathbf{A}) \cong \hom_{\Ab_k}(\mathbf{S}^\infty\mathbf{G}^{ab},\mathbf{A}).
	\]
	Proposition~\ref{prop:hogadichoudhury-RENAMED} implies that $\mathbf{S}^\infty\mathbf{G}^{ab}$ is strongly $\aone$-invariant, and since it is abelian it is thus strictly $\aone$-invariant.  We then conclude that
	\[
	\hom_{\Grp^{\aone}_k}(\mathbf{G},\mathbf{A}) \cong \hom_{\Ab^{\aone}_k}(\mathbf{S}^\infty\mathbf{G}^{ab},\mathbf{A}),
	\]
	i.e.., that $\mathbf{S}^\infty\mathbf{G}^{ab}$ and $\mathbf{G}^{ab}_{\aone}$ satisfy the same universal property, which is what we wanted to show.
\end{proof}

\begin{rem}
	\label{rem:firsthomologyasabelianization}
	If $k$ is a field and $\mathbf{G}$ is a strongly $\aone$-invariant sheaf of groups, then $\mathbf{G}^{ab}_{\aone} \cong \H_1^{\aone}(B\mathbf{G})$ by Morel's Hurewicz Theorem ~\cite[Theorem 6.35]{MField}.  If $k$ is perfect field and $\mathbf{G}$ is very strongly $\aone$-invariant, then $\H_1^{\aone}(B\mathbf{G}) \cong \mathbf{S}^{\infty}\mathbf{G}^{ab}$ by the preceding result.  
\end{rem}

\begin{rem}
	Analyzing the proof of Lemma~\ref{lem:abelianization} we observe that for any very strongly $\A^1$-invariant sheaf $\mathbf G$ there is always an epimorphism $\mathbf{G}^{ab} \to \mathbf{G}^{ab}_{\aone}$ and that this map is an isomorphism if $\mathbf{G}^{ab}$ happens to be $\aone$-invariant.  We do not know whether $\mathbf{G}^{ab}$ is $\aone$-invariant in general, though this does happen: see Example~\ref{ex:lowervsupper}.
\end{rem}

\begin{ex}[Contraction does not commute with abelianization]
	\label{ex:contractionandabelianization}
	In general, abelianization does not commute with contraction for an $\aone$-nilpotent sheaf of groups.  For example, one knows that $\bpi := \bpi_1(\pone)$ is a central extension of $\gm{}$ by $\K^{MW}_2$ by \cite[Theorem 7.29]{MField}.  In that case, $\H_1^{\aone}(\pone) \cong \K^{MW}_1$ by the suspension isomorphism, i.e., $\bpi^{ab} \cong \K^{MW}_1$.  We know that $(\K^{MW}_1)_{-1} \cong \K^{MW}_0$.  On the other hand, the contraction of $\bpi_1(\pone)$ is already abelian and isomorphic to $\K^{MW}_1 \oplus \Z$ by appeal to \cite[Corollary 7.34]{MField} and therefore coincides with its abelianization.   
\end{ex}

\begin{defn}
	If $\mathbf{G}$ is a very strongly $\aone$-invariant sheaf of groups, then we define $[\mathbf{G},\mathbf{G}]_{\aone}$ to be the kernel of the epimorphism $\mathbf{G} \to \mathbf{G}^{ab}_{\aone}$.
\end{defn}

\begin{rem}
	Assume $k$ is a perfect field.  If $\mathbf{G}$ is very strongly $\aone$-invariant, it is not clear that one can iterate the construction of the $\aone$-commutator subgroup sheaf to define an $\aone$-derived series because $[\mathbf{G},\mathbf{G}]_{\aone}$ is not evidently very strongly $\aone$-invariant.  In particular, it is not clear that analogs of the various equivalent characterizations of solvable groups hold for their motivic variants.  
\end{rem}

The next result is an analog for strictly $\aone$-invariant sheaves of a result of Stallings \cite[Theorem 2.1]{StallingsCS}.

\begin{lem}
	\label{lem:inductivestepLCC}
	Suppose $1 \to \mathbf{N} \to \mathbf{G} \to \mathbf{Q} \to 1$ is a short exact sequence of strongly $\aone$-invariant sheaves of groups.
	\begin{enumerate}[noitemsep,topsep=1pt]
		\item There exists a strictly $\aone$-invariant sheaf $\mathbf{N}/[\mathbf{N},\mathbf{G}]_{\aone}$ under $\mathbf{N}$, which is initial among strictly $\aone$-invariant sheaves under $\mathbf{N}^{ab}$ on which $\mathbf{G}$ acts trivially.
		\item There is an exact sequence of strictly $\aone$-invariant sheaves of the form
		\[
		\H_2^{\aone}(B\mathbf{G}) \longrightarrow \H_2^{\aone}(B\mathbf{Q}) \longrightarrow \mathbf{N}/[\mathbf{N},\mathbf{G}]_{\aone} \longrightarrow \mathbf{G}^{ab}_{\aone} \longrightarrow \mathbf{Q}^{ab}_{\aone} \longrightarrow 0.
		\]
		\item The exact sequence of the previous point is functorial in morphisms of short exact sequences.
		\item If $\mathbf N$ is very strongly $\A^1$-invariant then $\mathbf{N}/[\mathbf{N},\mathbf{G}]_{\aone} \simeq \mathbf{S}^{\infty}\mathbf{N}/[\mathbf{N},\mathbf{G}]$.
	\end{enumerate}
\end{lem}

\begin{rem}
If $\mathbf N$ is not very strongly $\A^1$-invariant then it is not clear to us if $\mathbf N \to \mathbf{N}/[\mathbf{N},\mathbf{G}]_{\aone}$ needs to be an epimorphism; the notation may be somewhat misleading in this case.
\end{rem}

\begin{proof}
(1)-(3).
	We establish the three points simultaneously.  Set $\mathcal{C} = \cof(B\mathbf{G} \to B\mathbf{Q})$.  Since $B\mathbf{N}$ is $0$-connected and $B\mathbf{Q}$ is $0$-connected, the canonical map $\Sigma B\mathbf{N} \to \mathscr{C}$ has $1$-connected fibers by Proposition~\ref{prop:blakersmassey}(2).  Since $\Sigma B\mathbf{N}$ is $1$-connected, $\mathscr{C}$ is $1$-connected as well.  The long exact sequence in homology takes the form (use Remark~\ref{rem:firsthomologyasabelianization})
	\[
	\H_2^{\aone}(B\mathbf{G}) \longrightarrow \H_2^{\aone}(B\mathbf{Q}) \longrightarrow \H_2^{\aone}(\mathscr{C}) \longrightarrow \mathbf{G}^{ab}_{\aone} \longrightarrow \mathbf{Q}^{ab}_{\aone} \longrightarrow 0,
	\]
	and moreover the map $\H_2^{\aone}(\Sigma B\mathbf{N}) \to \H_2^{\aone}(\mathscr{C})$ is an epimorphism.  
	
	Morel's Hurewicz theorem \cite[Theorem 6.37]{MField} implies that there is an isomorphism $\H_2^{\aone}(\Sigma B\mathbf{N}) \cong \H_1^{\aone}(B\mathbf{N}) \cong \mathbf{N}^{ab}_{\aone}$.  The relative Hurewicz theorem~\ref{thm:relativeHurewicz} implies that $\H_2^{\aone}(\mathscr{C})$ is the initial strictly $\aone$-invariant sheaf under $\mathbf{N}$ on which $\mathbf{G}$ acts trivially.  Combining the points above, we write $\mathbf{N}/[\mathbf{N},\mathbf{G}]_{\aone}$ for this sheaf.

(4). As in the proof of Lemma \ref{lem:abelianization}, it follows from Theorem \ref{thm:sinfty} that $\mathbf{S}^{\infty}\mathbf{N}/[\mathbf{N},\mathbf{G}]$ is the initial $\A^1$-invariant sheaf under $\mathbf N$ with a trivial action by $\mathbf G$, and this sheaf is strongly (hence strictly) $\A^1$-invariant by Proposition \ref{prop:hogadichoudhury-RENAMED}.  It thus coincides with $\mathbf{N}/[\mathbf{N},\mathbf{G}]_{\aone}$ by comparison of universal properties.
\end{proof}

\begin{construction}
	If $\mathbf{N}$ is a very strongly $\A^1$-invariant normal subgroup sheaf of $\mathbf{G}$ fitting into a short exact sequence $1 \to \mathbf{N} \to \mathbf{G} \to \mathbf{Q} \to 1$,
	we define $[\mathbf{N},\mathbf{G}]_{\aone}$ as the kernel of the epimorphism
	\[
	\mathbf{N} \longrightarrow \mathbf{N}/[\mathbf{N},\mathbf{G}]_{\aone}.
	\]
	Note that $[\mathbf{N},\mathbf{G}]_{\aone}$ is strongly $\aone$-invariant  by \cite[Lemma 3.1.14]{AFHLocalization}(1).
\end{construction}

To iterate this construction and define an $\A^1$-lower central series, we need to know that $[\mathbf{N}, \mathbf G]_{\aone}$ is again very strongly $\aone$-invariant.  This is automatic if $\mathbf G$ is locally nilpotent, by Remark \ref{rmk:extremely-strongly-aone-inv}.  We restrict our definition to this case.

\begin{defn}
	\label{defn:lowercentralseries}
	Let $\mathbf{G}$ be a strongly $\aone$-invariant, locally nilpotent sheaf of groups.
	The $\aone$-lower central series $\Gamma^i_{\aone}\mathbf{G}$, $i \geq 1$ is defined inductively by setting $\Gamma^1_{\aone}\mathbf{G} = \mathbf{G}$ and $\Gamma^{i}_{\aone}\mathbf{G} := [\Gamma_{\aone}^i \mathbf{G},\mathbf{G}]_{\aone}$.
\end{defn}

\begin{proposition}
	\label{prop:functoriallowercentralseries}
	The following statements hold.
	\begin{enumerate}[noitemsep,topsep=1pt]	
		\item The $\aone$-lower central series is functorial, i.e., if $f: \mathbf{G} \to \mathbf{G}'$ is a morphism of strongly $\aone$-invariant sheaves of groups, there are induced morphisms $\Gamma^i_{\aone}(f): \Gamma^i_{\aone}\mathbf{G} \to \Gamma^i_{\aone}\mathbf{G}'$ for all $i \geq 0$.
		\item A strongly $\aone$-invariant sheaf of groups $\mathbf{G}$ is $\aone$-nilpotent of nilpotence class $\leq c$ if and only if the $\aone$-lower central series has length $c$.
	\end{enumerate}
\end{proposition}

\begin{proof}
	For the first statement, observe that abelianization is functorial by Lemma~\ref{lem:abelianization} while the lower central series is built inductively by appeal to the exact sequence of Lemma~\ref{lem:inductivestepLCC}(2), which is functorial by Lemma~\ref{lem:inductivestepLCC}(3).  The second statement is established exactly as in the classical setting: an $\aone$-lower central series is an $\aone$-central series in the sense of \cite[Definition 3.2.1]{AFHLocalization}, and the $\aone$-lower central series is the $\aone$-central series of minimal length.
\end{proof}

\begin{ex}
	\label{ex:lowervsupper}
	The $\aone$-lower central series of $\bpi_1^{\aone}(\pone)$ differs from the upper central series (see \cite[Proposition 3.1.22]{AFHLocalization} for discussion of the latter).  Indeed, $\bpi_1(\pone)$ has center $\K^{MW}_2$ and $\aone$-nilpotence class $2$.  In contrast, $\H_1^{\aone}(\pone) \cong \K^{MW}_1$ and the epimorphism $\bpi_1(\pone) \to \H_1^{\aone}(\pone)$ is studied in \cite[\S 4]{HogadiChoudhury}: the kernel is the actual commutator subgroup of $\bpi_1(\pone)$.  
\end{ex}

The following result is a motivic analog of a presumably well-known fact about nilpotent groups, though we did not manage to locate a reference.

\begin{proposition}
	\label{prop:nilpotent-epi}
Let $\mathbf G' \to \mathbf G$ be a morphism of strongly $\A^1$-invariant sheaves of groups, with $\mathbf G$ locally $\A^1$-nilpotent\tom{Note that locally nilpotent + strongly $\A^1$-invariant does not seem to imply locally $\A^1$-nilpotent...} and $\mathbf G'$ very strongly $\A^1$-invariant.
If $\mathbf G'^{ab}_{\A^1} \to \mathbf G^{ab}_{\A^1}$ is surjective, then so is $\mathbf G' \to \mathbf G$.
\end{proposition}

\begin{proof}
By Lemma \ref{lemm:very-strongly-permanence}(1), we we may replace $\mathbf G'$ by its image $\mathbf H$ in $\mathbf G$.

We begin with a preparatory observation.   Assume $\mathbf K \subset \mathbf G$ is a strongly $\A^1$-invariant subsheaf of groups.  For $i \ge 1$ let $\mathbf K_i \subset \mathbf G$ be the subsheaf generated by $\Gamma^i_{\A^1} \mathbf G$ and $\mathbf K$.  Since $\Gamma^i_{\aone}\mathbf{G}$ is normal in $\mathbf{G}$, there are exact sequences of the form
\[
1 \longrightarrow \Gamma^i_{\aone}\mathbf{G} \longrightarrow \mathbf{K}_i \longrightarrow \mathbf{K}/(\mathbf{K} \cap \Gamma^i_{\A^1} \mathbf G) \longrightarrow 1.
\]
Since $\mathbf{K}/(\mathbf{K} \cap \Gamma^i_{\A^1} \mathbf G)$ is the image of $\mathbf{K}$ under the quotient morphism $\mathbf G \to \mathbf G/\Gamma^i_{\A^1}$, it is strongly $\A^1$-invariant ($\mathbf K$ being very strongly $\A^1$-invariant by Remark \ref{rmk:extremely-strongly-aone-inv}). Thus $\mathbf{K}_i$ is strongly $\aone$-invariant by appeal to \cite[Lemma 3.1.14(2)]{AFHLocalization}.

Next, we claim that $\mathbf K_{i+1} \subset \mathbf K_i$ is again normal.  Since $\Gamma^{i+1}_{\A^1} \mathbf G \subset \mathbf G$ is normal, it suffices to prove (abusing notation slightly) that if $k \in \mathbf K$ and $g \in \Gamma^i_{\A^1} \mathbf G$ then $k^g \in \mathbf K_{i+1}$.  This statement holds since
\[ 
k^g = kk^{-1}g^{-1}kg = k[k^{-1},g^{-1}] \in \mathbf K \cdot \Gamma^{i+1}_{\A^1} \mathbf G, 
\] 
the filtration being central.

Now we apply this construction with $\mathbf K := \mathbf H$; we shall prove by induction that $\mathbf H_i = \mathbf G$ for all $i$.
Since also \[ \mathbf H = \bigcap_i \mathbf H_i \] by local nilpotency (if $s^*$ is any stalk functor, then $s^* \mathbf H \subset s^*(\bigcap_i \mathbf H_i) \subset s^* \mathbf H_i$ for any $i$, but by assumption $s^* \mathbf H_i = s^* \mathbf H$ for $i$ large enough), this will conclude the proof.
Noting that $\mathbf H_1 = \mathbf G$ by definition, we assume that $\mathbf H_n = \mathbf G$.  The discussion above shows that $\mathbf H_{n+1} \subset \mathbf H_n = \mathbf G$ is normal; set $\mathbf Q := \mathbf{H}_{n}/\mathbf{H}_{n+1}$.  The sheaf $\mathbf{H}_{n+1}$ is very strongly $\aone$-invariant (Remark~\ref{rmk:extremely-strongly-aone-inv}) and the sheaf $\mathbf{Q}$ is strongly $\aone$-invariant by appeal to Lemma~\ref{lem:a1invquotient}.  

Since $\mathbf H^{ab}_{\A^1}$ surjects onto $\mathbf G^{ab}_{\A^1} \simeq (\mathbf H_n)^{ab}_{\A^1}$, also $(\mathbf H_{n+1})^{ab}_{\A^1} \to (\mathbf H_{n})^{ab}_{\A^1}$ is surjective.  The functor $(-)^{ab}_{\A^1}$ is a left adjoint and hence preserves quotients; thus $\mathbf Q^{ab}_{\A^1} = 0$.  Note also that $\mathbf{Q}$ is necessarily locally $\aone$-nilpotent as a quotient of the locally $\aone$-nilpotent sheaf of groups $\mathbf{G}$ \cite[Lemma 3.2.7(2)]{AFHLocalization}.  Since $\mathbf{Q}^{ab}_{\aone}$ is trivial and $\mathbf{Q}$ is locally $\aone$-nilpotent, one may inductively deduce that $\mathbf{Q} = 0$. Thus, we conclude that $\mathbf{H}_{n+1} = \mathbf{G}$ as well.
\end{proof}

The next result is an analog of \cite[Lemma 3.1]{StallingsCS}.

\begin{proposition}
	\label{prop:stallings}
	If $f: \mathbf{G}_1 \to \mathbf{G}_2$ is a morphism of strongly $\aone$-invariant, locally nilpotent sheaves of groups such that $\H_1^{\aone}(Bf)$ is an isomorphism and $\H_2^{\aone}(Bf)$ is an epimorphism, then 
	\[
	\mathbf{G}_1/\Gamma^{i+1}_{\aone}\mathbf{G}_1 \longrightarrow \mathbf{G}_2/\Gamma^{i+1}_{\aone}\mathbf{G}_2
	\]
	is an isomorphism for all $i \geq 1$.  In particular, if $\mathbf{G}_1$ and $\mathbf{G}_2$ are locally $\A^1$-nilpotent, then $f$ is an isomorphism.
\end{proposition}

\begin{proof}
	Taking $\mathbf{G} = \mathbf{G}_i$, $\mathbf{N} = \Gamma^r_{\aone}\mathbf{G}_i$, Lemma~\ref{lem:inductivestepLCC} in conjunction with the definition of the $\aone$-lower central series yields a morphism of commutative diagrams of the form
	\[
	\xymatrix{
	\H_2^{\aone}(B\mathbf{G}_1) \ar[r]\ar[d] & \H_2^{\aone}(B \mathbf{G}_1/\Gamma^r_{\aone}\mathbf{G}_1) \ar[r]\ar[d] & \Gamma^r_{\aone}\mathbf{G}_1/\Gamma^{r+1}_{\aone}\mathbf{G}_1 \ar[r]\ar[d] & \H_1^{\aone}(B\mathbf{G}_1) \ar[r]\ar[d] & \H_1^{\aone}(B\mathbf{G}_1/\Gamma^r_{\aone}\mathbf{G}_1)\ar[d] \\
		\H_2^{\aone}(B\mathbf{G}_2) \ar[r] & \H_2^{\aone}(B \mathbf{G}_2/\Gamma^r_{\aone}\mathbf{G}_2) \ar[r] & \Gamma^r_{\aone}\mathbf{G}_2/\Gamma^{r+1}_{\aone}\mathbf{G}_2 \ar[r] & \H_1^{\aone}(B\mathbf{G}_2) \ar[r] & \H_1^{\aone}(B\mathbf{G}_2/\Gamma^r_{\aone}\mathbf{G}_2)  \\
	}
	\]
	Assume inductively that $\mathbf{G}_1/\Gamma^r_{\aone}\mathbf{G}_1 \to \mathbf{G}_2/\Gamma^r_{\aone}\mathbf{G}_2$ is an isomorphism; the base case $r = 2$ follows from the assumption that $\H_1(Bf)$ is an isomorphism.  In that case, the second and fifth vertical arrows from the left are isomorphisms.  By hypothesis, the first vertical arrow is an epimorphism and the fourth vertical arrow is an isomorphism by assumption.  The five lemma implies that the third vertical arrow is an isomorphism as well.  Then, functoriality of the lower central series in conjunction with the short exact sequences
	\[
	\xymatrix{
	1 \ar[r] & \Gamma^r_{\aone}\mathbf{G}_i/\Gamma^{r+1}_{\aone}\mathbf{G}_i \ar[r] & \mathbf{G}_i/\Gamma^{r+1}_{\aone} \mathbf{G}_i \ar[r] & \mathbf{G}_i/\Gamma^{r}_{\aone} \mathbf{G}_i \ar[r] & 1 
	}
	\]
	together with the five lemma allow us to conclude that $\mathbf{G}_1/\Gamma^{r+1}_{\aone} \mathbf{G}_1 \to \mathbf{G}_2/\Gamma^{r+1}_{\aone} \mathbf{G}_2$ is an isomorphism.  
\end{proof}

More generally, we can refine the $\aone$-lower central series of a strongly $\aone$-invariant sheaf $\mathbf{G}$ to an $\aone$-lower central series for an action.  

\begin{construction}
	\label{const:aoneinvariantpicentralseries}
	Suppose $\bpi$ is a strongly $\aone$-invariant sheaf of groups and $\mathbf{G}$ is a very strongly $\aone$-invariant sheaf with $\bpi$-action.  In that case, there is a largest quotient sheaf of groups $\mathbf{G}_{\bpi}$ on which $\bpi$ acts trivially.  The sheaf $\mathbf{G}_{\bpi}$ need not be $\aone$-invariant.  However, Theorem~\ref{thm:sinfty} implies that 
	\[
	\mathbf{G} \longrightarrow \mathbf{G}_{\bpi} \longrightarrow \mathbf{S}^{\infty}\mathbf{G}_{\bpi}
	\]
	is an epimorphism with $\mathbf{S}^{\infty}\mathbf{G}_{\bpi}$ $\aone$-invariant. Very strong invariance of $\mathbf G$ implies that $\mathbf{S}^{\infty}\mathbf{G}_{\bpi}$ is strongly $\aone$-invariant.  It is thus necessarily the initial strongly $\aone$-invariant quotient of $\mathbf{G}$ with trivial $\bpi$-action.  In that case, we define
	\[
	\Gamma^{2}_{\bpi}\mathbf{G} := \ker(\mathbf{G} \to \mathbf{S}^{\infty}\mathbf{G}_{\bpi}),
	\]
	which is strongly $\aone$-invariant by appeal to \cite[Lemma 3.1.14(1)]{AFHLocalization}. If $\mathbf G$ is locally nilpotent, then for $n \geq 3$, we inductively define 
	\[
	\Gamma^{n}_{\bpi}\mathbf{G} := \Gamma^2_{\bpi}\Gamma^{n-1}_{\bpi}\mathbf{G}.
	\]
	By definition, $\Gamma^i_{\bpi}\mathbf{G}$ is a $\bpi$-central series for $\mathbf{G}$ in the sense of \cite[Definition 3.2.1]{AFHLocalization}.
\end{construction}

\begin{proposition}
	\label{prop:functorialpicentralseries}
	Suppose $\mathbf{G}$ is a strongly $\aone$-invariant, locally nilpotent sheaf with action of $\bpi$.  
	\begin{enumerate}[noitemsep,topsep=1pt]
		\item The series $\Gamma^i_{\bpi}\mathbf{G}$ is functorial in $\bpi$ and $\mathbf{G}$.
		\item When $\bpi = \mathbf{G}$ acting by conjugation on $\mathbf{G}$, $\Gamma^i_{\mathbf{G}}\mathbf{G}$ coincides with the $\aone$-lower central series of $\mathbf{G}$ of \textup{Definition~\ref{defn:lowercentralseries}}.
		\item The action of $\bpi$ on $\mathbf{G}$ is $\aone$-nilpotent if and only if $\Gamma^i_{\bpi}\mathbf{G}$ terminates after finitely many steps.
	\end{enumerate}
\end{proposition}

\begin{proof}
	Under the assumptions, $\mathbf{G}$ is very strongly $\aone$-invariant by appeal to Proposition~\ref{prop:lnilpstrongimpliesverystrong}.  Granted that observation, the first statement is immediate from the fact that $\mathbf{S}^{\infty}$ is functorial (Theorem~\ref{thm:sinfty}(1)).  The second statement follows from the fact that the corresponding statement holds for the classical central series.  The final statement is immediate from the definitions.
\end{proof}

\begin{entry}
	\label{par:aonecoinvariants}
	Suppose $\mathbf{A}$ is a strictly $\aone$-invariant sheaf with an action of a strongly $\aone$-invariant sheaf of groups $\bpi$.  We may consider the twisted Eilenberg-Mac Lane space $K^{\bpi}(\mathbf{A},n)$ for any integer $n \geq 2$.  In that case, there is a fiber sequence
	\[
	K(\mathbf{A},n) \longrightarrow K^{\bpi}(\mathbf{A},n) \stackrel{f}{\longrightarrow} B\bpi.
	\]
	The relative Hurewicz theorem~\ref{thm:relativeHurewicz} tells us that $\H_{n}^{\aone}(\cof(f))$ is the largest strictly $\aone$-invariant quotient of $\mathbf{A}$ on which $\bpi$ acts trivially.  We write $\H_0^{\aone}(\bpi,\mathbf{A})$ for this strictly $\aone$-invariant sheaf, which we call the sheaf of $\aone$-coinvariants of the $\bpi$-action on $\mathbf{A}$.  In particular if the action of $\bpi$ on $\mathbf{A}$ is trivial, then $\H_0^{\aone}(\bpi;\mathbf{A}) = \mathbf{A}$.  
\end{entry}

\begin{proposition}
	\label{prop:functorialityofaonecoinvariants}
	Assume $\mathbf{A}$ is a $\bpi$-module.
	\begin{enumerate}[noitemsep,topsep=1pt]
		\item The sheaf $\H_0^{\aone}(\bpi;\mathbf{A})$ is functorial in $\bpi$ and $\mathbf{A}$.
		\item The sheaf $\H_0^{\aone}(\bpi;\mathbf{A})$ coincides with $\mathbf{S}^{\infty}\mathbf{A}_{\bpi}$ from \textup{Construction~\ref{const:aoneinvariantpicentralseries}}.
	\end{enumerate}
\end{proposition}

\begin{proof}
	The first statement is immediate from the definitions.  The second statement follows because $\H_0^{\aone}(\bpi;\mathbf{A})$ and $\mathbf{S}^{\infty}\mathbf{A}_{\bpi}$ have the same universal property: they are the largest strictly $\aone$-invariant quotients of $\mathbf{A}$ with trivial action of $\bpi$.
\end{proof}

\begin{rem}
	Another example arises from the mod $p$ lower central series (or Stallings filtration) of a group \cite{StallingsCS}.  Recall that if $p$ is a fixed integer, $G$ is a group and $U$ is a subgroup, then $G \sharp_p U$ is the subgroup generated by $[g,u]v^p$ for $g \in G, u,v \in U$.  If $U$ is normal in $G$, then $G \sharp_p U$ is normal in $G$.  One then defines the mod $p$ lower central series inductively by setting $\Gamma^1_pG = G$, $\Gamma^{i+1}_p G := G \sharp_p \Gamma^i_pG$.  The construction above yields a central series which for $p = 0$ coincides with the lower central series, and for $p$ a prime is the largest descending central series with subquotients that are $\mathbb{F}_p$-vector spaces.  Since the above construction is evidently functorial it makes sense for presheaves of groups and then for sheaves of groups by sheafication.  In that case, we can define the mod $p$-$\aone$-lower central series following the procedure above.  If $\mathbf{G}$ is a strongly $\aone$-invariant, locally nilpotent sheaf of groups, we define $\Gamma^i_{\aone,p}\mathbf{G}$ by setting $\Gamma^1_{\aone,p}\mathbf{G} = \mathbf{G}$ and then inductively defining
	\[
	\Gamma^{i+1}_{\aone,p}\mathbf{G} := \ker(\mathbf{G} \longrightarrow \mathbf{S}^{\infty}\mathbf{G}/\mathbf{G} \sharp_p \Gamma^i_{\aone,p} \mathbf{G}).
	\]
	Theorem~\ref{thm:sinfty} and Proposition~\ref{prop:hogadichoudhury-RENAMED} then imply that $\Gamma^{i+1}_{\aone,p}\mathbf{G}$ is a functorial decreasing central series with successive subquotients that are strictly $\aone$-invariant sheaves of ${\mathbb F}_p$-vector spaces.
\end{rem}

\section{Principal refinements of Moore--Postnikov factorizations}
Recall that a pointed, connected motivic space $\mathscr{X}$ is called nilpotent if $\bpi_1(\mathscr{X})$ is $\aone$-nilpotent and $\bpi_1(\mathscr{X})$ acts $\aone$-nilpotently on higher homotopy sheaves \cite[Definition 3.3.1]{AFHLocalization}.  More generally if $f: \mathscr{E} \to \mathscr{B}$ is a morphism of pointed connected motivic spaces, then $f$ is nilpotent if $\fib(f)$ is connected and the action of $\bpi_1(\mathscr{E})$ on $\bpi_i(\fib(f))$ is $\aone$-nilpotent for all $i \geq 1$.  The relative Hurewicz theorem~\ref{thm:relativeHurewicz} in conjunction with the functorial central series for nilpotent actions from Proposition~\ref{prop:functorialpicentralseries} allows us to construct a {\em functorial} principal refinement of the Moore-Postnikov tower improving \cite[Theorem 3.3.13]{AFHLocalization}.

\begin{theorem}
	\label{thm:nilpotentprincipalrefinement}
	Assume $k$ is a perfect field.  Suppose $f: \mathscr{E} \to \mathscr{B}$ is a morphism of pointed, connected motivic spaces.  If $f$ is nilpotent, and $\bpi_1(\fib(f))$ is very strongly $\aone$-invariant, then $f$ admits a functorial principal refinement.  More precisely, if $f': \mathscr{E}' \to \mathscr{B}$ is another nilpotent morphism of pointed, connected motivic spaces with $\bpi_1(\fib(f'))$ very strongly $\aone$-invariant, and $g: \mathscr{E} \to \mathscr{E}'$ is a morphism of spaces over $\mathscr{B}$, then $g$ induces homomorphisms 
	\[
	\Gamma^i(g): \Gamma^i_{\bpi_1(\mathscr{E})}\bpi_j(\fib(f)) \longrightarrow \Gamma^i_{\bpi_1(\mathscr{E}')}\bpi_j(\fib(f'))
	\] 
	for all $i,j$ which induce morphisms of layers of the principal refinements of $f$ and $f'$.	
\end{theorem}

\begin{proof}
	Repeat the proof of \cite[Corollary 4.2.4]{AFHLocalization} using the $\aone$-lower central series for the action of $\bpi_1(\mathscr{E})$ on $\fib(f)$ of Construction~\ref{const:aoneinvariantpicentralseries}.  Functoriality follows from Proposition~\ref{prop:functorialpicentralseries} upon replacing appeals to \cite[Theorem 4.2.1]{AFHLocalization} with appeals to Theorem~\ref{thm:relativeHurewicz} (in view of Proposition~\ref{prop:functorialityofaonecoinvariants}).
\end{proof}




With the above preparations at hand, we can now establish the analog of the Whitehead theorem for nilpotent motivic spaces extending \cite[Theorem 4.2.2]{AFHLocalization}.  In order to formulate the next theorem, write $\SH^{S^1}(k)$ for the stabilization of $\ho{k}$ in the sense of \cite[\S1.4]{HA}.  We write $\Sigma^{\infty}_{S^1}: \ho{k}_* \adj \SH^{S^1}(k): \Omega^\infty$ for the canonical adjunction.  The functor $\Sigma^\infty_{S^1}$ preserves cofiber sequences, being a left adjoint.

\begin{theorem}
	\label{thm:whitehead}
Assume $k$ is a perfect field.  Suppose $f: \mathscr{X} \to \mathscr{Y}$ is a (pointed) morphism of nilpotent motivic spaces.  If $n \geq 0$ is an integer, the following statements are equivalent.
\begin{enumerate}[noitemsep,topsep=1pt] 
	\item The map $f$ has $(n-1)$-connected fibers.
	\item The map $\Sigma^{\infty}_{S^1}f$ has $(n-1)$-connected fiber.
	\item The maps $\H_{i}^{\aone}(f)$ are isomorphisms for $i < n$ and $\H_{n}^{\aone}(f)$ is an epimorphism.
\end{enumerate}
 \end{theorem}

\begin{proof}
The implications (1) $\Longrightarrow$ (2) $\Longrightarrow$ (3) follow by appeal to Proposition~\ref{prop:cofiberconnectivity}(1).  For the reverse implications we proceed as follows. Assume (3); we shall prove (1).

The statement is immediate if $n = 0$ since nilpotent spaces are connected by assumption.  If $n=1$ we must prove that $\bpi_1(f)$ is surjective knowing only that $\bpi_1(f)^{ab}_{\A^1}$ is surjective; this is precisely Proposition \ref{prop:nilpotent-epi}.

We therefore assume $n \ge 2$, in which case we know that $\bpi_1(f)$ is an epimorphism by what we just observed.  Then, by appeal to Theorem~\ref{thm:nilpotentprincipalrefinement} we know that the Moore--Postnikov factorization of $f$ admits a principal refinement.  In other words, there are fiber sequences 
\[ 
\mathscr Y_{j+1} \longrightarrow \mathscr Y_j \longrightarrow K(\mathbf A_j, n_j), \]
with $n_j$ a weakly increasing sequence of integers, $n_j \ge 2$, $\mathscr Y_0 = \mathscr Y$ and such that the connectivity of $\mathscr X \to \mathscr Y_N$ is tends to infinity as $N$ tends to $\infty$.  Without loss of generality, we may also assume that $\mathbf A_j \ne 0$.

Essentially by construction, $n_0 = t+1$ if and only if $\bpi_t \fib(f)$ is the first non-vanishing homotopy group. Our goal is thus to prove that $n_0 \ge n+1$.
Suppose, in the interest of obtaining a contradiction that $n_0 \le n$.  The long exact homology sequence for the cofiber of $\mathscr Y_{j+1} \to \mathscr Y_j$ together with the relative Hurewicz Theorem \ref{thm:relativeHurewicz} yield (using that the action of $\bpi_1 \mathscr X$ on $\mathbf A_j$ is trivial) 
\[ 
\H_{n_0}^{\A^1}(\mathscr Y_{j+1}) \longrightarrow \H_{n_0}^{\A^1}(\mathscr Y_j) \longrightarrow \mathbf A_j' \longrightarrow \H_{n_0-1}^{\A^1}(\mathscr Y_{j+1}) \longrightarrow \H_{n_0-1}^{\A^1}(\mathscr Y_j) \longrightarrow 0; 
\] 
where $\mathbf A_j' = \mathbf A_j$ if $n_j = n_0$ and $\mathbf A_j' = 0$ otherwise.
We deduce that $\H_{n_0-1}^{\A^1}(\mathscr Y_{j+1}) \to \H_{n_0-1}^{\A^1}(\mathscr Y_{j})$ is surjective for every $j$.
Since $\mathscr X \to \mathscr Y_N$ is highly connected and we have already established that (1) implies (3), we find that $\H_{n_0-1}^{\A^1}(\mathscr X) \to \H_{n_0-1}^{\A^1}(\mathscr Y_N) \to \H_{n_0-1}^{\A^1}(\mathscr Y_1)$ is surjective.
On the other hand since $n_0 \le n$, $\H_{n_0-1}^{\A^1}(\mathscr X) \to \H_{n_0-1}^{\A^1}(\mathscr Y)$ is an isomorphism, whence also $\H_{n_0-1}^{\A^1}(\mathscr Y_1) \weq \H_{n_0-1}^{\A^1}(\mathscr Y)$.
Finally since $\H_{n_0}^{\A^1}(\mathscr X) \to \H_{n_0}^{\A^1}(\mathscr Y)$ is surjective so is $\H_{n_0}^{\A^1}(\mathscr Y_1) \to \H_{n_0}^{\A^1}(\mathscr Y)$, whence we get (from the above six term exact sequence) 
\[ 
0 \longrightarrow \mathbf A_0 \longrightarrow \H_{n_0-1}^{\A^1}(\mathscr Y_1) \longrightarrow \H_{n_0-1}^{\A^1}(\mathscr Y) \longrightarrow 0. 
\]
Since the second map is an isomorphism as shown above, we learn that $\mathbf A_0 = 0$, which is the desired contradiction.
\end{proof}

\begin{rem}
	Assume $\mathscr{X}$ is a pointed connected motivic space with $\bpi_1(\mathscr{X}) =: \bpi_1$.  The map $\mathscr{X} \to B\bpi_1$ has $1$-connected fibers, and the argument for Theorem~\ref{thm:whitehead} implies the motivic analog of a result of Hopf \cite[Satz II]{Hopffundgroup} that $\H_2^{\aone}(\mathscr{X}) \to \H_2^{\aone}(B\bpi_1)$ is an epimorphism.
\end{rem}  

\begin{rem}
	Of course, the Whitehead theorem for nilpotent motivic spaces could also be deduced following Bousfield \cite[Lemma 8.9]{Bousfieldhomologylocalization}.  Consider Morel's $\aone$-homology theory $\H^{\aone}_*$.  We can localize $\ho{k}$ with respect to this homology theory; following \cite[Definition 4.3.1]{AFHLocalization} this localization yields the $\Z$-$\aone$-local homotopy category.  A morphism $f: \mathscr{X} \to \mathscr{Y}$ is an $\H^{\aone}_*$-local equivalence if it induces an isomorphism after applying $\H^{\aone}_*$.  Theorem~\ref{thm:whitehead} implies that nilpotent motivic spaces are $\H^{\aone}_*$-local.  
	
	Ignoring base-point issues, one could check the converse as follows.  First, observe that $K(\mathbf{A},n)$ is an $\H^{\aone}_*$-local space for any $n \geq 0$.  Indeed, this follows from the existence of the universal coefficient spectral sequence (see \cite[Proof of Proposition 4.1.2]{AFHLocalization}). In that case, any limit of $\H^{\aone}_{*}$-local spaces is $\H^{\aone}_*$-local.  The existence of a principal refinement of the Postnikov tower then implies that nilpotent motivic spaces are $\H^{\aone}_*$-local.
\end{rem}

\subsection*{Aside: Moore--Postnikov towers for locally $\A^1$-nilpotent morphisms}
In \cite[Definition 3.3.1]{AFHLocalization}, the notion of a locally $\A^1$-nilpotent morphism of motivic spaces was also introduced: these are morphisms $f: \mathscr{E} \to \mathscr{B}$ such that $\fib(f)$ is connected, and the action of $\bpi_1(\mathscr{E})$ on $\bpi_1(\fib(f))$ is locally $\aone$-nilpotent for all $i \geq 1$.  The functorial principal refinements of Moore--Postnikov factorizations described above can be extended to locally $\A^1$-nilpotent morphisms of motivic spaces by limiting processes; we briefly explain this here.  

Let us revisit the construction of principal refinements of nilpotent morphisms.
Suppose $f: \mathscr{E} \to \mathscr{B}$ is a morphism of pointed, connected motivic spaces such that $\bpi_1(f)$ is an epimorphism.  Set $\mathscr F = \fib(f)$.  Set $\mathscr E_i = \tau_{\le i} \mathscr E$, where the Postnikov truncation is taken in the topos $\Shv_\Nis(\Sm_k)_{/\mathscr B}$.
Viewed as an object of $\Shv_\Nis(\Sm_k)$ (i.e., ignoring the map to $\mathscr B$), $\mathscr E_i$ is just $\tau_{\le i} \mathscr F$.\todo{ref?}
The object $\mathscr E_{i+1} \in \Shv_\Nis(\Sm_k)_{/\mathscr E_i}$ is an $(i+1)$-gerbe and hence (at least if $i > 1$) classified by a map $* \to K_{\mathscr E_i}(\mathbf A_{i+1}, i+2) \in \Shv_\Nis(\Sm_k)_{/\mathscr E_i}$ \cite[Theorem 7.2.2.26]{HTT}, for some $\mathbf A_{i+1} \in \mathrm{Ab}(\Shv_\Nis(\Sm_k)_{/\mathscr E_i})$.\todo{say something about $i=1$?}
Concretely, $\mathbf A_{i+1} = \bpi_{i+1} \mathscr F$, with its action by $\bpi_1 \mathscr E_i \weq \bpi_1 \mathscr E$.

Suppose given a (strongly $\A^1$-invariant) quotient $\mathbf A_{i+1,0} := \mathbf A_{i+1} \to \mathbf Q_{i+1,1}$.  Let $\mathscr E_{i,1}$ be the fiber of the composite $\mathscr E_{i,0} := \mathscr E_i \to K_{\mathscr E_i}(\mathbf Q_{i+1,1}, i+2)$.  Then there is a canonical map $\mathscr E_{i+1} \to \mathscr E_{i,1}$ exhibiting $\mathscr E_{i+1} \in \Shv_\Nis(\Sm_k)_{/\mathscr E_{i,1}}$.
This is another $(i+1)$-gerbe, this time banded by $\mathbf A_{i+1,1} := \ker(\mathbf A_{i+1,0} \to \mathbf Q_{i+1,1})$, and hence classified by a map $* \to K_{\mathscr E_{i,1}}(\mathbf A_{i+1,1}, i+2) \in \Shv_\Nis(\Sm_k)_{/\mathscr E_{i,1}}$.  Provided we are given a decreasing filtration of $\mathbf A_{i+1}$, we can keep repeating this process.  Assuming furthermore that the action of $\bpi_1(\mathscr E)$ on $\mathbf Q_{i+1,j}$ is trivial, we obtain principalized fiber sequences 
\[ 
\mathscr E_{i,j+1} \longrightarrow \mathscr E_{i,j} \longrightarrow K(\mathbf Q_{i+1,j}, i+2) \in \Spc(k)_*. 
\]
If $f$ is nilpotent, then the filtrations can be chosen to be finite, and so $\mathscr E_{i,j} \weq \mathscr E_{i+1}$ for $j \gg 0$.  This way $\mathscr E$ is built from $\mathscr B$ by a sequence of principal fibrations.

Now suppose that $f$ is only locally $\A^1$-nilpotent.  Then all steps in this process can still be performed.  The only issue is that we may not have $\mathscr E_{i,j} \weq \mathscr E_{i+1}$ for any finite $j$.  However, we still have the following.

\begin{proposition} 
	\label{prop:generalized-postnikov}
Let $f: \mathscr E \to \mathscr B$ be a locally $\A^1$-nilpotent morphism of motivic spaces.  Using notation as above, we have 
\[ 
\mathscr E_i \weq \lim_j \mathscr E_{i,j}. 
\]
\end{proposition}
\begin{proof}
Let $X \in \Sm_k$ be connected.
It will suffice to show the equivalence after restricting to the small Zariski site $X_{\Zar}$ (note that this restriction preserves limits).
By assumption the filtration on the generic stalk of $X$ is of finite length.
Since the $\mathbf A_{i,j}$ are unramified \cite[Corollary 6.9]{MField} we deduce that $\mathbf A_{i,j}|_{X_{\Zar}} = 0$ for $j \gg 0$.
Hence $\mathscr E_{i+1}|_{X_{\Zar}} \weq \mathscr E_{i,j}|_{X_{\Zar}}$ for such $j$.
This proves the claim.
\end{proof}

\begin{cor}
Locally $\A^1$-nilpotent motivic spaces lie in the subcategory of $\Spc(k)_*$ generated under (iterated) limits by $S^1$-infinite loop spaces.
In particular, if $f: \mathscr X \to \mathscr Y \in \Spc(k)_*$ is a morphism of locally $\A^1$-nilpotent motivic spaces such that $\Sigma^\infty_{S^1} f$ is an equivalence, then $f$ is an equivalence.
\end{cor}
\begin{proof}
The first statement is clear by considering the construction of Proposition \ref{prop:generalized-postnikov} for $\mathscr X \to *$.
For the second statement, by Yoneda, it suffices to show that for any locally $\A^1$-nilpotent motivic space $\mathscr Z$ we have $\Map_*(\mathscr X, \mathscr Z) \weq \Map_*(\mathscr Y, \mathscr Z)$.
The class of spaces $\mathscr Z$ such that this holds is closed under (iterated) limits and contains all $S^1$-infinite loop spaces.
Hence we conclude by the first statement.
\end{proof}

\section{Applications} \tom{I have not revised this}
In this section, we describe some applications of the Whitehead theorem building on a comment made in the introduction: in checking that a morphism between spaces is an $S^1$-stable equivalence, singular spaces may intercede.  In order to make sense of this principle, we begin by recalling how to define $S^1$-stable homotopy types of singular (ind-schemes) in characteristic $0$.  Then, we analyze the main example: the Schubert filtration on the affine Grassmannian for $SL_2$.  Using this filtration we are able to deduce Theorem~\ref{thmintro:exoticequivalence}, which provides an exotic equivalence in unstable motivic homotopy theory. 

In what follows, if $\mathscr{X}$ is a pointed space, then we write $\Omega^{p,q}\mathscr{X}$ for the internal (pointed) mapping space $\iMap_*(S^{p,q},\mathscr{X})$.  Before getting to the more complicated case mentioned above, we describe a simple motivating example.

\begin{ex}
	Consider the space ${\mathbb P}^{\infty} \cong B\gm{}$ \cite[\S 4 Proposition 3.8]{MV}.  In fact, ${\mathbb P}^{\infty} = \colim_n {\mathbb P}^n$ and in particular there is an inclusion $\pone \hookrightarrow {\mathbb P}^{\infty}$ of the ``bottom cell".  Under the equivalence $\pone \cong S^{2,1}$, this map is adjoint to a map
	\[
	S^1 \longrightarrow \Omega^{1,1} {\mathbb P}^{\infty}.
	\]
	The space ${\mathbb P}^{\infty}$ has a commutative $h$-space structure that is inherited by $\Omega^{1,1} {\mathbb P}^{\infty}$ and the displayed map is an $h$-map for this structure.  In fact, the above map is an equivalence by explicit computation.  Indeed, $\Omega^{1,1} {\mathbb P}^{\infty}$ is connected, $1$-truncated, and an explicit computation shows that $\bpi_1(\Omega^{1,1}{\mathbb P}^{\infty}) \cong (\K^M_1)_{-1} \cong \Z$ and the displayed map induces an isomorphism on $\bpi_1$ and is thus an equivalence.
\end{ex}

\subsubsection*{$S^1$-stable homotopy types of singular varieties}
Write $\Sch^{ft}_k$ for the category of {\em finite type} $k$-schemes and consider the inclusion functor $\Sm_k \hookrightarrow \Sch^{ft}_k$ (which is fully-faithful).  There is an induced restriction functor $R: \mathrm{P}(\Sch^{ft}_k) \to \mathrm{P}(\Sm_k)$, where $\mathrm{P}(\Sch^{ft}_k)$ is the $\infty$-category of presheaves of spaces on $\Sch^{ft}_k$.  We would like to show $\mathscr{X} \in \mathrm{P}(\Sch^{ft}_k)$ has a well-defined homotopy type.  To make this precise, recall that one may equip $\Sch^{ft}_k$ with the structure of a site as follows.  Consider the cd-structure \cite[Definition 2.1.1]{AHW} on $\Sch^{ft}_k$ generated by Nisnevich distinguished squares and squares arising from diagrams of the form $Z \stackrel{i}{\to} X \stackrel{p}{\leftarrow} Y$ where $i$ is a closed immersion, $p$ is an isomorphism on the complement of $i$.  The cd-structure just described defines a topology on $\Sch^{ft}_k$ called the cdh-topology.

Recall that $\mathscr{X} \in \mathrm{P}(\Sch^{ft}_k)$ is {\em cdh-excisive} if $\mathscr{X}(\emptyset)$ is contractible and if $\mathscr{X}$ takes squares in the cd-structure described above to cartesian squares.  We define $\mathrm{H}^{ft}(k)$ to be the subcategory of $\mathrm{P}(\Sch^{ft}_k)$ spanned by presheaves of spaces that are cdh-excisive and $\aone$-invariant (see the beginning of Section~\ref{ss:motiviclocalization}).  As before there is a motivic localization functor 
\[
\mathrm{L}_{mot}^{cdh}: \mathrm{P}(\Sch^{ft}_k) \longrightarrow \mathrm{H}^{ft}(k);
\]
that preserves finite products.  

The restriction functor sends $cdh$-excisive and $\aone$-invariant sheaves to Nisnevich and $\aone$-invariant sheaves on $\Sm_k$ and one thus has a diagram of functors, which need not commute:
\[
\xymatrix{
 \mathrm{P}(\Sch^{ft}_k) \ar[r]^{R}\ar[d]^{\mathrm{L}_{mot}^{cdh}} & \mathrm{P}(\Sm_k) \ar[d]^{\mathrm{L}_{mot}} \\
  \mathrm{H}^{ft}(k) \ar[r] & \ho{k}.
}
\]
One can stabilize both sides with respect to the simplicial circle.  We write $\SH^{S^1}_{ft}(k)$ for the category obtained by inverting $S^1$ on $\mathrm{H}^{ft}(k)_*$ (as described in \cite[\S 4.1]{Robalo}).  Abusing terminology slightly, we write $\mathrm{L}_{mot}^{cdh}\mathscr{X}_+$ for the functor $\mathrm{P}(\Sch^{ft}_k) \to \SH^{S^1}_{ft}(k)$ and $\mathrm{L}_{mot}\mathscr{X}_+$ for the functor $\mathrm{P}(\Sm_k) \to \SH^{S^1}_{ft}(k)$.  

This functor $R$ has a left adjoint $E: \mathrm{P}(\Sm_k) \to \mathrm{P}(\Sch^{ft}_k)$.  The above inclusion functor corresponds to a morphism of sites
\[
\pi: (\Sch^{ft}_k)_{cdh} \longrightarrow (\Sm_k)_{\Nis}.
\]
Voevodsky established that the functor $\pi^*: \SH^{S^1}_{cdh}(k) \to \SH^{S^1}(k)$ is an equivalence \cite[Theorem 4.2]{VNisvscdh}.  The stable motivic homotopy type of $\mathscr{X} \in \mathrm{P}(\Sch^{ft}_k)$ is then defined to be $\pi^* \mathrm{L}_{mot}^{cdh}\mathscr{X}_+$.  The following result establishes that singular schemes have well-defined stable motivic homotopy types (cf. \cite[Proposition 17]{Bachmanngrassmannian}).

\begin{proposition}
	\label{prop:s1stablehomotopytypesofsingularschemes}
	Assume $k$ has characteristic $0$.  
	\begin{enumerate}[noitemsep,topsep=1pt]
		\item For any $\mathscr{X} \in \mathrm{P}(\Sch^{ft}_k)$, there is an equivalence 
	\[\mathrm{L}_{mot} R\mathscr{X}_+ \cong \pi^* \mathrm{L}_{mot}^{cdh}\mathscr{X}_+.\]
	\item If $X = \colim_n X_n$ is a filtered colimit of finite type schemes, then $\mathrm{L}_{mot}RX_+ \cong \colim_n \mathrm{L}_{mot}^{cdh} (X_n)_+$. 
	\end{enumerate}
\end{proposition}

\begin{proof}
	Since the functor $\Sm_k \hookrightarrow \Sch^{ft}_k$ is fully-faithful, the composite $\mathrm{P}(\Sm_k) \stackrel{E}{\to} \mathrm{P}(\Sch^{ft}_k) \stackrel{R}{\to} \mathrm{P}(\Sm_k)$ is the identity functor on $\mathrm{P}(\Sm_k)$.  It follows that $RER \mathscr{X}_+ \cong R\mathscr{X}_+$, i.e., if $\eta: ER \to id$ is the counit of adjunction, then $R(\eta)$ is an equivalence.  On the other hand, since $k$ has characteristic $0$, resolution of singularities implies that all schemes are $cdh$-locally smooth \cite[Lemma 4.3]{VNisvscdh}.  Thus, the assertion that $R(\eta)$ is an equivalence is sufficient to guarantee that $\eta$ becomes an equivalence after applying $\mathrm{L}_{mot}^{cdh}(-)_+$, which is what we wanted to show. 
	
	The second statement follows from the first since all functors under consideration commute with filtered colimits.
\end{proof}

\begin{rem}
	\label{rem:positivechar}
	Assume that $k$ is a field with characteristic $p$ and assume $\ell$ is a prime different from $p$.  If strong resolution of singularities was known in positive characteristic, then Proposition~\ref{prop:s1stablehomotopytypesofsingularschemes} would extend to that situation as well.  Unfortunately, we are not aware of a version of Proposition~\ref{prop:s1stablehomotopytypesofsingularschemes} that holds, even assuming we invert the characteristic exponent of $p$, say working with the $\ell dh$-topology of \cite[Definition 2.1.11]{Kellyldh}, even though finite-type singular $k$-schemes are $\ell dh$-locally smooth \cite[Corollary 2.1.15]{Kellyldh}, this is not sufficient to guarantee $\ell dh$-descent for suspension spectra.  
\end{rem}

\subsubsection*{Affine Grassmannian models of Tate loop spaces}
Suppose $G$ is a split reductive group scheme (which we will always assume to be $SL_2$ below).  One model of the affine grassmannian $\op{Gr}_G$ is an ind-scheme representing the fppf sheafification of the functor on $\Sch_k$ defined by 
\[
U \mapsto G(\Gamma(U,\mathcal{O}_U)((t)))/G(\Gamma(U,\mathcal{O}_U)[[t]]);
\]  
see \cite[\S 2]{BeauvilleLaszlo} for more details in the case of interest for us (i.e., $G = SL_2$).  Quillen \cite{MitchellQuillen} and Garland--Raghunathan \cite{GarlandRaghunathan} observed that $\op{Gr}_G(\cplx)$ had the structure of an associative $h$-space: if $K$ is the maximal compact subgroup of $G(\cplx)$, then there is a homotopy equivalence $\Omega K \cong \op{Gr}_G$.  The following result can be viewed as a lift of these results to the motivic homotopy category and equips $\op{Gr}_G$ with an associative $h$-space structure in $\ho{k}$.

\begin{theorem}[{\cite[Theorem 15]{Bachmanngrassmannian}}]
	\label{thm:affinegrassmannianmodel}
If $k$ is a field, and $G$ is a split reductive group scheme, then there is an equivalence
\[
\Omega^{1,1} G \cong \op{Gr}_{G}.
\]
\end{theorem}

As mentioned above, the affine Grassmannian of $\op{Gr}_G$ has the structure of an ind-scheme: the Schubert cell decomposition explicitly allows us to see $\op{Gr}_G$ as a filtered colimit of projective $k$-schemes that are frequently singular.   The Schubert cell description in the case of $SL_2$ is particularly simple and was elucidated and made extremely explicit in the work of Mitchell \cite{MitchellSUn}, who also analyzed the interaction of this filtration (which he called the Bott filtration) with the multiplicative structure induced by the weak equivalence $\op{Gr}_G(\cplx) \cong \Omega K$ mentioned above.

\begin{theorem}[Mitchell]
	\label{thm:mitchellfiltration}
	The ind-scheme $\op{Gr}_{SL_2}$ admits a Schubert filtration, i.e., there exists a directed sequence of projective schemes $F_{2,r}$ such that $\op{Gr}_{SL_2} = \colim_r F_{2,r}$ and such that the schemes $F_{2,r}$ have the following properties.
	\begin{enumerate}[noitemsep,topsep=1pt]
		\item The scheme $F_{2,1}$ is isomorphic to $\pone$.
		\item The complement $F_{2,r} \setminus F_{2,r-1}$ is isomorphic to an affine space of dimension $r$ and $F_{2,r}/F_{2,r-1}$ is equivalent to $S^{2r,r}$.
		\item The multiplication for the $h$-space structure induces a morphism $F_{2,1}^{\times r} \to F_{2,r}$ that is a resolution of singularities.
		\item The inclusion $F_{2,1}^{\times r-1} \hookrightarrow F_{2,1}^{\times r} \to F_{2,r} \to S^{2r,r}$ (as the first $r-1$-factors) is the constant map to the base-point.
	\end{enumerate}
\end{theorem}

\begin{proof}
	Mitchell states these results only in the case $k = \cplx$, but the proofs work over any characteristic $0$ field.  For a construction of the relevant varieties as schemes see \cite[Proposition 2.6]{BeauvilleLaszlo}.  With those observations in mind, these results are a summary of \cite[2.7-8, 2.11-12]{MitchellSUn} in the case $n = 2$.
\end{proof}

\begin{rem}
	The map in Theorem~\ref{thm:mitchellfiltration}(3) is an affine Bott--Samelson--Demazure resolution (see \cite[p. 920]{BottSamelson} and \cite[\S 3]{Demazure} for a scheme-theoretic treatment).  The maximal compact subgroup of $SL_2(\cplx)$ is $SU(2) \cong S^3$.  Under the identification $\Omega S^3 = \Omega \Sigma S^2 \cong J(S^2)$, the space $\Omega\Sigma S^2$ carries an increasing James filtration.  Under the weak equivalence $\Omega SU(2) \cong \op{Gr}_{SL_2}(\cplx)$, the James filtration is carried to the Schubert filtration.  Moreover, the inclusion $F_{2,1} \hookrightarrow \op{Gr}_{SL_2}$ realizes Bott's generating complex \cite[\S 5]{Bottloops}.
\end{rem}

\subsubsection*{An unstable exotic periodicity}
After Theorem~\ref{thm:affinegrassmannianmodel} and Theorem~\ref{thm:mitchellfiltration}(1) there is a morphism of (ind-)schemes
\[
\pone \longrightarrow \op{Gr}_{SL_2},
\]
corresponding to the inclusion $F_{2,1} \hookrightarrow \op{Gr}_{SL_2}$.  The right hand side has the structure has the structure of an $h$-group by means of Theorem~\ref{thm:affinegrassmannianmodel}, in particular, it is an associative monoid. 

Given any connected space $\mathscr{X}$, the James construction $J(\mathscr{X})$ \cite[Theorem 2.4.2]{AWW} comes equipped with a morphism $\mathscr{X} \to J(\mathscr{X})$ that is universal among maps from $\mathscr{X}$ to associative monoids:  if $\mathscr{M}$ is an associative monoid in $\ho{k}$, then given a map $f: \mathscr{X} \to \mathscr{M}$, $f$ factors uniquely through the map $\mathscr{X} \to J(\mathscr{X})$.  The universal property of the James construction applied to the map $\pone \to \op{Gr}_{SL_2}$ of the previous point factors through a morphism:
\[
\psi: J(\pone) \longrightarrow \op{Gr}_{SL_2}.
\]
Moreover, the left hand side has an increasing filtration by spaces of the form $J_n(\pone)$ with successive subquotients equivalent to $S^{2n,n}$; we refer to this filtration as the James filtration.  After a single simplicial suspension, the left hand side splits as $\bigvee_{n \geq 0} S^{2n+1,n}$ by \cite[Proposition 5.6]{WickelgrenWilliams}.  The next result, which implies Theorem~\ref{thmintro:exoticequivalence} essentially follows from assertion of  Theorem~\ref{thm:mitchellfiltration} that the James filtration coincides with the Schubert filtration; in the motivic context, we view the ``weight shifting" displayed by $\psi$ as an exotic periodicity.

\begin{theorem}
	\label{thm:exoticequivalence}
	Assume $k$ is a field having characteristic $0$.
	\begin{enumerate}[noitemsep,topsep=1pt]
		\item The morphism $\Sigma^{\infty}_{S^1}\psi$ is an equivalence.
		\item The morphism $\psi:  \Omega^{1,0}\Sigma^{1,0}S^{2,1} \to \Omega^{1,1}\Sigma^{1,1}S^{2,1}$ is an equivalence.
		\item The morphism $S^{3,1} \to \Omega^{1,1}\op{HP}^{\infty}$ of the introduction is a weak equivalence.
	\end{enumerate}
\end{theorem}

\begin{proof}
	Consider the Schubert filtration on $\op{Gr}_{SL_2}$ from Theorem~\ref{thm:mitchellfiltration}.  By definition of the James filtration on $J(\pone)$, the layers $J_r(\pone)$ are precisely the images of ${\pone}^{\times r}$ under the monoid structure.  Thus, Theorem~\ref{thm:mitchellfiltration}(3) implies that the product map factors through a map
	\[
	J_r(\pone) \longrightarrow F_{2,r}.
	\]
	By definition, the morphism in question is an equivalence when $r = 1$.  Theorem~\ref{thm:mitchellfiltration}(4) implies that the induced maps $J_r(\pone)/J_{r-1}(\pone) \to F_{2,r}/F_{2,r-1}$ are equivalences.  An induction argument using Proposition~\ref{prop:s1stablehomotopytypesofsingularschemes}(1) implies that the maps $J_r(\pone) \to F_{2,r}$ are $S^1$-stable equivalences.  In that case, Proposition~\ref{prop:s1stablehomotopytypesofsingularschemes}(2) again in conjunction with Theorem~\ref{thm:mitchellfiltration} implies that the map $J(\pone) \to \op{Gr}_{SL_2}$ is an $S^1$-stable equivalence.  The second point follows from the first by appeal to Theorem~\ref{thm:whitehead} since both spaces are connected $h$-spaces (in particular, nilpotent).
	
	The map $\psi$ is an $h$-map by construction and therefore $B(\psi)$ is an equivalence.  In that case, we have
	\[
	S^{3,1} \cong B\Omega S^{3,1} \cong B(J(\pone)) \cong B(\op{Gr}_{SL_2}) \cong B(\Omega^{1,1}SL_2) \cong \Omega^{1,1}BSL_2.	
	\]
	The inclusion of the bottom cell $\op{HP}^1 \to \op{HP}^{\infty}$ is classified by the Hopf bundle $\nu$, which is a $GW(k)$-module generator of $\bpi_{4,2}(BSL_2) \cong GW(k)$.  Up to multiplication by a unit, the adjoint of this map coincides with the map inducing the equivalence in the previous point.  Replacing the equivalence of the previous point by its product with inverse of the unit in $GW(k)$ just described if necessary, the result from the introduction follows.
\end{proof}

\subsubsection*{An exceptional EHP sequence}
The combinatorial James--Hopf invariant (in this context, \cite[Definition 3.1.3]{AWW}) defines a morphism:
\[
\Omega^{1,0} \Sigma^{1,0} S^{2,1} \longrightarrow \Omega^{1,0} \Sigma^{1,0} S^{4,2}.
\]
We can identify the $2$-local fiber of this map in a fashion that extends the EHP fiber sequence of Wickelgren--Williams \cite{WickelgrenWilliams}.  The following result uses the full strength of Theorem~\ref{thm:whitehead}.

\begin{proposition}
	\label{prop:exceptionalEHP}
	Over any field $k$, there is $2$-local homotopy fiber sequence of motivic spaces of the form
	\[
	S^{2,1} \longrightarrow \Omega^{1,0} S^{3,1} \longrightarrow \Omega^{1,0} S^{5,2}.
	\]
\end{proposition}

\begin{proof}
	The James--Hopf invariant described above is a map of pointed connected and simple spaces so its fiber is nilpotent by appeal to \cite[Theorem 3.3.6]{AFHLocalization}.  The composite map $S^{2,1} \to \Omega^{1,0}S^{5,2}$ is null because the target is simply connected.  A choice of null-homotopy then determines a map $f: \pone \to \fib(\Omega^{1,0}S^{3,1} \to \Omega^{1,0} S^{5,2})$.  The space $S^{2,1}$ is a nilpotent motivic space by \cite[Theorem 3.4.8]{AFHLocalization} (take $G = SL_2$ in that statement).  Now, let us localize all spaces and maps at $2$; via \cite[Theorem 4.3.9]{AFHLocalization} there is a well-behaved $2$-localization for nilpotent motivic spaces.  
	
	To check that $f$ is an equivalence after $2$-localization, by the Whitehead theorem for nilpotent motivic spaces \ref{thm:whitehead}, it suffices to check this is the case after $S^1$-suspension.  In that case, the proof of \cite[Theorem 8.3]{WickelgrenWilliams}, in conjunction with \cite[Proposition 4.7]{WickelgrenWilliams} shows that $\Sigma^{\infty} f$ is a $2$-local equivalence, as claimed.
\end{proof}

\begin{rem}
	The connected motivic spheres $S^{q+1,q}$ are only known to be $h$-spaces when $q = 0$ and $q = 2$, so the above argument cannot be generalized much further.  
\end{rem}

Consider the map $S^{4,2} \to K(\Z(2),4)$ classifying a generator.  This map is adjoint to maps $S^{3,1} \to K(\Z(1),3)$, $S^{3,2} \to K(\Z(2),3)$ and $S^{2,1} \to K(\Z(1),2)$ and moreover these adjoint maps are also the relevant fundamental class maps.  We define 
\[
S^{3,i} \langle 3 \rangle := \fib(S^{3,i} \to K(\Z(i),3)).  
\]
We obtain analogs of Cohen's exceptional $2$-local fiber sequence \cite[Theorem 3.3]{Cohen}. 

\begin{cor}
	\label{cor:cohenexceptional}
	Assume $k$ is a field.  There is a $2$-local fiber sequence of the form
	\[
	S^{3,2} \longrightarrow \Omega^{1,0} S^{3,1} \langle 3 \rangle \longrightarrow \Omega^{1,0} S^{5,2}
	\]
	and, if $k$ has characteristic $0$, then there is a fiber sequence of the form
	\[
	S^{3,2} \longrightarrow \Omega^{1,1} S^{3,2} \langle 3 \rangle \longrightarrow \Omega^{1,0} S^{5,2}.
	\]
	In both cases, the left hand map is induced by a suitable generator of $\bpi_{3,2}(\Omega^{1,i}S^{3,i} \langle i \rangle)$.
\end{cor}

\begin{proof}
	There is a commutative square of the form:
	\[
	\xymatrix{
		\Omega^{1,0} S^{3,1} \ar[r]\ar[d] & \Omega^{1,0} S^{5,3} \ar[d] \\
		K(\Z(1),2) \ar[r] & \ast
	}
	\]
	where the left vertical map is obtained by looping the fundamental class map $S^{3,1} \to K(\Z(1),3)$.  Proposition~\ref{prop:exceptionalEHP} describes the ($2$-local) fiber of the top row, and the unwinding the definition of the map $\Omega^{1,0}S^{3,1} \to K(\Z(1),2)$ we conclude that the induced map of fibers is the fundamental class map $S^{2,1} \to K(\Z(1),2)$.  By definition of $\eta$, one identifies $\fib(S^{2,1} \to K(\Z(1),2)) \cong S^{3,2}$ and the induced map as $\eta$.  Taking vertical fibers yields the first statement.
	
	For the second statement, assuming $k$ has characteristic $0$, consider the exotic equivalence $\Omega^{1,0} \Sigma^{1,0} S^{2,1} \to \Omega^{1,1}\Sigma^{1,1} S^{2,1}$ of Theorem~\ref{thm:exoticequivalence}.  Under that equivalence, there is an induced map $\Omega^{1,1}S^{3,2} \to K(\Z(1),2)$ that coincides with the map obtained by applying $\Omega^{1,1}$ to the fundamental class $S^{3,2} \to K(\Z(2),3)$.  Repeating the argument of the preceding paragraph with this new map yields the second fiber sequence.
\end{proof}



\begin{footnotesize}
\bibliographystyle{alpha}
\bibliography{freudenthalref}

\begin{thebibliography}{AWW17}

\bibitem[ABH23]{ABHFreudenthal}
A.~Asok, T.~Bachmann, and M.J. Hopkins.
\newblock On {${\mathbb P}^1$}-stabilization in unstable motivic homotopy
  theory.
\newblock 2023.
\newblock {\em Available at} \url{https://arxiv.org/abs/2306.04631}.

\bibitem[AF16]{AFComparison}
A.~Asok and J.~Fasel.
\newblock Comparing {E}uler classes.
\newblock {\em Q. J. Math.}, 67(4):603--635, 2016.

\bibitem[AFH22]{AFHLocalization}
A.~Asok, J.~Fasel, and M.~J. Hopkins.
\newblock Localization and nilpotent spaces in {$\mathbb A^1$}-homotopy theory.
\newblock {\em Compos. Math.}, 158(3):654--720, 2022.

\bibitem[AHW17]{AHW}
A.~Asok, M.~Hoyois, and M.~Wendt.
\newblock Affine representability results in {$\mathbb A^1$}-homotopy theory,
  {I}: vector bundles.
\newblock {\em Duke Math. J.}, 166(10):1923--1953, 2017.

\bibitem[AM11]{AM}
A.~Asok and F.~Morel.
\newblock Smooth varieties up to {$\mathbb A^1$}-homotopy and algebraic
  {$h$}-cobordisms.
\newblock {\em Adv. Math.}, 227(5):1990--2058, 2011.

\bibitem[AWW17]{AWW}
A.~Asok, K.~Wickelgren, and T.B. Williams.
\newblock The simplicial suspension sequence in {$\mathbb{A}^1$}-homotopy.
\newblock {\em Geom. Topol.}, 21(4):2093--2160, 2017.

\bibitem[Bac19]{Bachmanngrassmannian}
T.~Bachmann.
\newblock Affine {G}rassmannians in {$\mathbb A^1$}-homotopy theory.
\newblock {\em Selecta Math. (N.S.)}, 25(2):Paper No. 25, 14, 2019.

\bibitem[BHS15]{BHS}
C.~Balwe, A.~Hogadi, and A.~Sawant.
\newblock {$\mathbb{A}^1$}-connected components of schemes.
\newblock {\em Adv. Math.}, 282:335--361, 2015.

\bibitem[BL94]{BeauvilleLaszlo}
A.~Beauville and Y.~Laszlo.
\newblock Conformal blocks and generalized theta functions.
\newblock {\em Comm. Math. Phys.}, 164(2):385--419, 1994.

\bibitem[Bot58]{Bottloops}
R.~Bott.
\newblock The space of loops on a {L}ie group.
\newblock {\em Michigan Math. J.}, 5:35--61, 1958.

\bibitem[Bou75]{Bousfieldhomologylocalization}
A.~K. Bousfield.
\newblock The localization of spaces with respect to homology.
\newblock {\em Topology}, 14:133--150, 1975.

\bibitem[BS58]{BottSamelson}
R.~Bott and H.~Samelson.
\newblock Applications of the theory of {M}orse to symmetric spaces.
\newblock {\em Amer. J. Math.}, 80:964--1029, 1958.

\bibitem[CH22]{HogadiChoudhury}
U.~Choudhury and A.~Hogadi.
\newblock The {H}urewicz map in motivic homotopy theory.
\newblock {\em Annals of K-Theory}, 7(1):179--190, jun 2022.

\bibitem[CMN79]{CMN}
F.~R. Cohen, J.~C. Moore, and J.~A. Neisendorfer.
\newblock The double suspension and exponents of the homotopy groups of
  spheres.
\newblock {\em Ann. of Math. (2)}, 110(3):549--565, 1979.

\bibitem[Coh87]{Cohen}
F.~R. Cohen.
\newblock A course in some aspects of classical homotopy theory.
\newblock In {\em Algebraic topology ({S}eattle, {W}ash., 1985)}, volume 1286
  of {\em Lecture Notes in Math.}, pages 1--92. Springer, Berlin, 1987.

\bibitem[Dem74]{Demazure}
M.~Demazure.
\newblock D\'{e}singularisation des vari\'{e}t\'{e}s de {S}chubert
  g\'{e}n\'{e}ralis\'{e}es.
\newblock {\em Ann. Sci. \'{E}cole Norm. Sup. (4)}, 7:53--88, 1974.

\bibitem[DF71]{DrorWhitehead}
E.~Dror-Farjoun.
\newblock A generalization of the {W}hitehead theorem.
\newblock In {\em Symposium on {A}lgebraic {T}opology ({B}attelle {S}eattle
  {R}es. {C}enter, {S}eattle, {W}ash., 1971)}, Lecture Notes in Math., Vol.
  249, pages 13--22. Springer, Berlin, 1971.

\bibitem[DH21]{DH}
S.~Devalapurkar and P.~Haine.
\newblock On the {J}ames and {H}ilton-{M}ilnor {S}plittings, and the metastable
  {EHP} sequence.
\newblock {\em Doc. Math.}, 26:1423--1464, 2021.

\bibitem[Gan65]{Ganea}
T.~Ganea.
\newblock A generalization of the homology and homotopy suspension.
\newblock {\em Comment. Math. Helv.}, 39:295--322, 1965.

\bibitem[Ger75]{GerstenWhitehead}
S.~M. Gersten.
\newblock The {W}hitehead theorem for nilpotent spaces.
\newblock {\em Proc. Amer. Math. Soc.}, 47:259--260, 1975.

\bibitem[GJ09]{GoerssJardine}
P.~G. Goerss and J.~F. Jardine.
\newblock {\em Simplicial homotopy theory}.
\newblock Modern Birkh\"{a}user Classics. Birkh\"{a}user Verlag, Basel, 2009.

\bibitem[GR75]{GarlandRaghunathan}
H.~Garland and M.~S. Raghunathan.
\newblock A {B}ruhat decomposition for the loop space of a compact group: a new
  approach to results of {B}ott.
\newblock {\em Proc. Nat. Acad. Sci. U.S.A.}, 72(12):4716--4717, 1975.

\bibitem[Hop42]{Hopffundgroup}
H.~Hopf.
\newblock Fundamentalgruppe und zweite {B}ettische {G}ruppe.
\newblock {\em Comment. Math. Helv.}, 14:257--309, 1942.

\bibitem[Hov99]{Hovey}
M.~Hovey.
\newblock {\em Model categories}, volume~63 of {\em Mathematical Surveys and
  Monographs}.
\newblock American Mathematical Society, Providence, RI, 1999.

\bibitem[Hoy14]{Hoyois}
M.~Hoyois.
\newblock A quadratic refinement of the {G}rothendieck-{L}efschetz-{V}erdier
  trace formula.
\newblock {\em Algebr. Geom. Topol.}, 14(6):3603--3658, 2014.

\bibitem[Hoy17]{HoyoisEquiv}
M.~Hoyois.
\newblock The six operations in equivariant motivic homotopy theory.
\newblock {\em Adv. Math.}, 305:197--279, 2017.

\bibitem[HW20]{HahnWilson}
J.~Hahn and D.~Wilson.
\newblock Eilenberg--{M}ac {L}ane spectra as equivariant {T}hom spectra.
\newblock {\em Geom. Topol.}, 24(6):2709--2748, 2020.

\bibitem[Kel17]{Kellyldh}
S.~Kelly.
\newblock Voevodsky motives and $l$dh descent.
\newblock {\em Ast{\'e}risque}, 391, 2017.

\bibitem[Koi22]{Koizumi}
J.~Koizumi.
\newblock Zeroth {$\mathbb A^1$}-homology of smooth proper varieties.
\newblock {\em New York J. Math.}, 28:824--834, 2022.

\bibitem[Lur09]{HTT}
J.~Lurie.
\newblock {\em Higher topos theory}, volume 170 of {\em Annals of Mathematics
  Studies}.
\newblock Princeton University Press, Princeton, NJ, 2009.

\bibitem[Lur17]{HA}
J.~Lurie.
\newblock {\em Higher {A}lgebra}.
\newblock 2017.
\newblock Available at \url{https://www.math.ias.edu/~lurie/papers/HA.pdf}.

\bibitem[Mat76]{Mather}
M.~Mather.
\newblock Pull-backs in homotopy theory.
\newblock {\em Canadian J. Math.}, 28(2):225--263, 1976.

\bibitem[Mit86]{MitchellSUn}
S.~A. Mitchell.
\newblock A filtration of the loops on {${\rm SU}(n)$} by {S}chubert varieties.
\newblock {\em Math. Z.}, 193(3):347--362, 1986.

\bibitem[Mit88]{MitchellQuillen}
S.~A. Mitchell.
\newblock Quillen's theorem on buildings and the loops on a symmetric space.
\newblock {\em Enseign. Math. (2)}, 34(1-2):123--166, 1988.

\bibitem[Mor05]{MStable}
F.~Morel.
\newblock The stable {${\mathbb A}^1$}-connectivity theorems.
\newblock {\em $K$-Theory}, 35(1-2):1--68, 2005.

\bibitem[Mor12]{MField}
F.~Morel.
\newblock {\em {$\mathbb A^1$}-algebraic topology over a field}, volume 2052 of
  {\em Lecture Notes in Mathematics}.
\newblock Springer, Heidelberg, 2012.

\bibitem[MV99]{MV}
F.~Morel and V.~Voevodsky.
\newblock $\mathbb{A}^1$-homotopy theory of schemes.
\newblock {\em Publications Mathématiques de l'Institut des Hautes Études
  Scientifiques}, 90(1):45--143, 1999.

\bibitem[Pup74]{Puppe}
V.~Puppe.
\newblock A remark on ``homotopy fibrations''.
\newblock {\em Manuscripta Math.}, 12:113--120, 1974.

\bibitem[PW21]{PaninWalter}
I.~Panin and C.~Walter.
\newblock Quaternionic {G}rassmannians and {B}orel classes in algebraic
  geometry.
\newblock {\em Algebra i Analiz}, 33(1):136--193, 2021.

\bibitem[Qui67]{QuillenHA}
D.~G. Quillen.
\newblock {\em Homotopical algebra}.
\newblock Lecture Notes in Mathematics, No. 43. Springer-Verlag, Berlin-New
  York, 1967.

\bibitem[Rez98]{Rezk}
C.~Rezk.
\newblock Fibrations and homotopy colimits of simplicial sheaves, 1998.
\newblock Available at \url{https://arxiv.org/abs/math/9811038}.

\bibitem[Rob15]{Robalo}
M.~Robalo.
\newblock {$K$}-theory and the bridge from motives to noncommutative motives.
\newblock {\em Adv. Math.}, 269:399--550, 2015.

\bibitem[Sta65]{StallingsCS}
J.~Stallings.
\newblock Homology and central series of groups.
\newblock {\em J. Algebra}, 2:170--181, 1965.

\bibitem[Str12]{Strunk}
F.~Strunk.
\newblock On motivic spherical bundles.
\newblock 2012.
\newblock {\em Thesis Institut f\"ur Mathematik Universit\"at Osnabr\"uck},
  available at
  \url{https://repositorium.uni-osnabrueck.de/bitstream/urn:nbn:de:gbv:700-2013052710851/3/thesis_strunk.pdf}.

\bibitem[Too76]{Toomer}
G.~H. Toomer.
\newblock Homology equivalences and a technique of {G}anea.
\newblock {\em Math. Z.}, 150(3):273--279, 1976.

\bibitem[Voe10]{VNisvscdh}
V.~Voevodsky.
\newblock Unstable motivic homotopy categories in {N}isnevich and
  cdh-topologies.
\newblock {\em J. Pure Appl. Algebra}, 214(8):1399--1406, 2010.

\bibitem[Wen11]{WendtI}
M.~Wendt.
\newblock Classifying spaces and fibrations of simplicial sheaves.
\newblock {\em J. Homotopy Relat. Struct.}, 6(1):1--38, 2011.

\bibitem[WW19]{WickelgrenWilliams}
K.~Wickelgren and T.B. Williams.
\newblock The simplicial {EHP} sequence in {$\mathbb A^1$}-algebraic topology.
\newblock {\em Geom. Topol.}, 23(4):1691--1777, 2019.

\end{thebibliography}
\end{footnotesize}
\Addresses
\end{document}